\documentclass[journal,twoside,web]{ieeecolor}
\usepackage{generic}
\usepackage{amsmath,amssymb,amsfonts,bm}
\usepackage{algorithmic}
\usepackage{graphicx}
\usepackage{algorithm,algorithmic}
\usepackage{hyperref}
\usepackage[thmmarks, amsmath]{ntheorem}
\let\labelindent\relax
\usepackage{enumitem}
\usepackage{mathtools}
\usepackage{textcomp}

\usepackage[style=ieee]{biblatex}
\theoremstyle{plain} 
\theoremnumbering{arabic} 
\theoremseparator{.} 
\newtheorem{theorem}{Theorem}
\newtheorem{lemma}{Lemma}

\theoremstyle{plain} 
\theoremheaderfont{\bf\itshape} 
\theorembodyfont{\itshape} 
\theoremnumbering{arabic} 
\theoremseparator{:} 
\newtheorem{definition}{Definition}
\newtheorem{assumption}{Assumption}

\theoremstyle{plain}
\theoremheaderfont{\itshape}  
\theorembodyfont{\normalfont} 
\theoremseparator{:}          
\newtheorem{example}{Example}
\newtheorem{remark}{Remark}

\usepackage[labelformat=simple]{subcaption}

\DeclareMathOperator*{\argmin}{arg\,min}
\newcommand\numberthis{\addtocounter{equation}{1}\tag{\theequation}}
\newcommand{\alglinelabel}{%
  \addtocounter{ALC@line}{-1}
  \refstepcounter{ALC@line}
  \label
}

\allowdisplaybreaks

\def\BibTeX{{\rm B\kern-.05em{\sc i\kern-.025em b}\kern-.08em
    T\kern-.1667em\lower.7ex\hbox{E}\kern-.125emX}}
\begin{document}
\title{Bridging Batch and Streaming Estimations to System Identification under Adversarial Attacks}
\author{Jihun Kim, \IEEEmembership{Graduate Student Member, IEEE} and Javad Lavaei, \IEEEmembership{Fellow, IEEE}
\thanks{This work was supported by the U. S. Army Research Laboratory and the U. S. Army Research Office under Grant W911NF2010219, Office of Naval Research under Grant N000142412673, AFOSR, NSF, and the UC Noyce Initiative.}
\thanks{Jihun Kim and Javad Lavaei are with the Department of Industrial Engineering and Operations Research, University of California, Berkeley, CA 94720 USA (emails:  {\tt\footnotesize \{jihun.kim, lavaei\}@berkeley.edu}).} 
\thanks{A preliminary version of this paper has been published in the 64$^{\text{th}}$ IEEE Conference on Decision and Control (CDC), Rio de Janeiro, Brazil, Dec 09-12, 2025 \cite{kim2025system}. The previous version primarily discussed the recovery of the Markov parameter matrix using the $\ell_1$-norm estimator, while this journal version demonstrates the superior performance of the $\ell_2$-norm estimator in terms of both recovery time and estimation error. Furthermore, we propose a stochastic subgradient descent algorithm for streaming data, particularly for scenarios where the recovery time required for the $\ell_2$-norm estimator to achieve accurate estimation has not yet been reached.}
}
\maketitle

\begin{abstract}
System identification in modern engineering systems faces emerging challenges from unanticipated adversarial attacks beyond existing detection mechanisms. In this work, we obtain a provably accurate estimate of the Markov parameter matrix of order $k$ to identify partially observed linear systems, in which the probability $p$ of having an attack at each time is smaller than $0.5$. We show that given the batch data accumulated up to time $T^*$, the $\ell_2$-norm estimator achieves an error with stepwise exponential decay as the attack probability $p$ decreases. Furthermore, low-dimensional attack structures can decouple the exponential decay order from the attack probability $p$, achieving an error decay that scales with $T^*$. We then propose a stochastic projected subgradient descent algorithm on streaming data that produces an estimate at each time $t<T^*$, in which case the expected estimation error proves to be the larger of $\mathcal{O}(k/\sqrt{t})$ and the baseline error of the $\ell_2$-norm estimator at $T^*$. This stochastic approach illustrates how non-smooth estimators can leverage first-order methods despite lacking recursive formulas. Finally, we integrate batch and streaming estimations to recover the Hankel matrix using the appropriate estimates of the Markov parameter matrix, which enables the synthesis of a robust adaptive controller based on the estimated balanced truncated model under adversarial attacks.
\end{abstract}


\section{Introduction}
\label{sec:intro}
System identification is the process of identifying the underlying dynamics and the observation model of a dynamical system, given the input and observed output trajectories. 
Modern engineering systems, such as autonomous vehicles and smart grids, are often so complex that a full analysis of the systems is infeasible, and their dynamics and observation model may be partially or completely unknown.
In this case, controlling the system without knowledge of the true system often leads to suboptimal policies and poor performance, particularly when the controllers are designed based on the estimated system. 

Emerging challenges to accurate system identification often lie in the adversarial nature of disturbances injected into the dynamics. Benign disturbances, particularly with zero mean and independence over time, disregards temporal correlations, history-dependent drifts, and adversarial spikes in disturbances. Existing approaches for these benign disturbances include Dynamic Programming (DP) \cite{bertsekas1995dynamic, kim2024landscape}, Kalman filtering \cite{kalman1960new, kalman1961new}, and statistical learning such as the least-squares method \cite{simchowitz2018learningwithout,   sarkar2019near,jedra2020finite}. However, real-world scenarios may involve adversarial and nonzero-mean disturbances rather than independent zero-mean disturbances. In this paper, we refer to such adversarial disturbances as an \textit{adversarial attack} since the adversary can model them maliciously by leveraging all historical information. In particular, these disturbances fall within the scope of robust system identification and can be addressed via non-smooth estimators such as the $\ell_2$-norm estimator \cite{yalcin2024exact, zhang2024exact} or the $\ell_1$-norm estimator \cite{kim2024prevailing, kim2026sharp}.

In this paper, we study the identification of the input-output mapping of a partially observed linear system under adversarial attacks; i.e.,  it aims to recover the Markov parameter matrix of order $k$, which consists of the first $k$ Markov parameters that map the $k$ most recent inputs to the observation at each time. 
 To \textit{focus on the theoretical impact of attacks from the input-output perspective},  
 we consider the additive attack model that captures a broad class of direct and indirect attack mechanisms, including  false data injection as well as communication channel corruption, actuator manipulation, and physical dynamic state compromises \cite{cardenas2008secure}. 
All these attacks are integrated as additive perturbations to the system's input-output map. Moreover, we assume zero observation noise and infrequent attack occurrences, which finds justification in critical cyber-physical domains such as power grids already equipped with advanced security mechanisms \cite{pasqualetti2013cyber}. Cyberattacks can broadly be classified as detectable or undetectable, where most attacks are detectable by well-designed detectors and are effectively nullified or corrected by controllers, and thus the system dynamics is not corrupted (see literature on the correction of the states under observation noise \cite{candes2005decoding, fawzi2014secure, shoukry2016event, pajic2017attack, li2023disturbance}).  However, undetectable attacks that are far less frequent than detectable attacks 
 occur when a strong adversary leverages complete knowledge of the system to craft a sophisticated and malicious attack. These so-called \textit{stealthy} attacks slip through the system while remain unnoticed.

 The preliminary version \cite{kim2025system} of this paper focused on the $\ell_1$-norm estimator to recover the Markov parameter matrix of order $k$ with the \textit{batch data}, where one waits until a sufficient amount of data has been collected before applying the estimator.
 In this journal paper, we will show that the $\ell_2$-norm estimator is not only also capable of identifying the Markov parameter matrix with the batch data, but also outperforms the $\ell_1$-norm estimator. Standard non-asymptotic analysis using batch data aims to show that a favorable property holds for $t\geq T^*$ for some finite time $T^*$, which we prove to be the case for the adopted $\ell_2$-norm estimator. However, this analysis does not provide any guarantees for $t<T^*$. Intuitively, sufficient signal in the data accumulates only after $T^*$;  estimates before $T^*$ may show unstable behavior, which necessitates an alternative method to recover estimates before we arrive at time $T^*$.  
 
To provide a convergence guarantee to identify the Markov parameter matrix  with the \textit{streaming data} for $t<T^*$, meaning that the estimate is sequentially updated as each new data arrives in contrast to the batch setting, we propose a projected subgradient descent algorithm using the subgradient of the $\ell_2$-norm estimator. 
There have been approaches to learning dynamical systems via the first-order method \cite{hardt2018gradient, oymak2019stochastic, jain2021sgdrer} to streaming data in different ways, while they all focused on using the gradient of the smooth least-squares objective. 
We note that classical subgradient descent methods for non-smooth objectives 
(see \cite{boyd2004convex, nesterov2004intro}) were leveraged in the work \cite{yalcin2026subgradient}  
for fully observed systems to establish guarantees on the estimation error, but only after $t\geq T^*$. The proposed streaming estimation approach overcomes this limitation and provides convergence guarantees on the estimation error for all $t$ when little is known about this matrix.
Furthermore, we use a stochastic subgradient descent algorithm \cite{bubeck2015convex, bottou2018optimization}, which computes a subgradient of the $\ell_2$-norm term only regarding the most recent observation and thus provides an unbiased approximation of the full subgradient of the $\ell_2$-norm estimator at time $t$, thereby greatly reducing computational complexity while still guaranteeing convergence in expectation.

\textbf{Contribution and Outline.} 
In Section \ref{sec:probform}, we formulate our problem. We summarize the overall structure and the contribution of each subsequent section as follows:
\begin{enumerate}[leftmargin=0.5cm]
    \item 
   In Section \ref{sec:unconstrained}, we demonstrate that the $\ell_2$-norm estimator outperforms the $\ell_1$-norm estimator in terms of both recovery time and estimation error. In particular, to identify the Markov parameter matrix of order $k$, when the attack probability $p$ at a given time is smaller than $\frac{1}{2s}$ for some $s\in\{1,\dots,k-1\}$,    
   the $\ell_1$-norm estimator achieves the error of $\mathcal{O}(\sqrt{r}\cdot \rho^{s})$,
    whereas the $\ell_2$-norm estimator improves this bound to $\mathcal{O}(\rho^{s})$, where  $r$ is the dimension of observations and $0\leq \rho<1$ is a contraction factor of the true system. Thus, the estimation error decays exponentially {as $p$ decreases.}
    Similarly, the recovery time $T^*$ for the $\ell_2$-norm estimator no longer depends on $r$, unlike the recovery time for the $\ell_1$-norm estimator. 
    \item {In Section \ref{sec:subspace}, we consider a low-dimensional subspace attack scheme where an attack is constrained to the column space of a known matrix.
    Such configurations naturally arise in cyber-physical systems where 
physical structures limit adversarial injections to a known set of actuator channels, and the active sensor deployment is fully characterized  \cite{xie2017secure, showkatbakhsh2020securing}. 
Unlike in the full-dimensional attack scheme, the decay exponent $s$ is decoupled from the attack probability $p$, allowing the error decay to scale directly with the time horizon for every $p < \min\{0.5, 1-\rho\}$.
}

    \item While the $\ell_2$-norm estimator achieves $\mathcal{O}(\rho^{s})$ after the finite time $T^*$, it does not provide the error achievable at times $t<T^*$. For this purpose, in Section \ref{sec:subg}, we leverage a  stochastic projected subgradient descent with the subgradient of the $\ell_2$-norm term only regarding the most recent observation. 
 We {focus on $p<\frac{1}{2(k-1)}$ case} and  demonstrate that the estimation error decreases progressively with time $t$ in expectation, which was not guaranteed in the naive $\ell_2$-norm estimator {under full-dimensional attack.} In particular, the estimation error is given by the larger of a polynomially decaying term in $t$ and an exponentially decaying term in $k$;
 we achieve an error of $\mathcal{O}(\max\{k/\sqrt{t},~k^{1/4}\rho^{\frac{k-1}{2}}\})$  at time $t$ when little is known about the true Markov parameter matrix. If prior information of the matrix is available, the polynomial term can be improved to an exponentially decaying term in $t$, yielding $\mathcal{O}(\max\{\exp(-t/k^2),\rho^{k-1}\})$.
\item In Section \ref{sec:framework}, we finally provide the estimation error of the Hankel matrix using the  stochastic projected subgradient descent algorithm (streaming estimation) when $t<T^*$, and the $\ell_2$-norm estimator  (batch estimation) when $t\geq T^*$. Bridging batch and streaming estimation approaches allows us to construct an appropriately estimated balanced truncated model at all time steps and synthesize the associated robust adaptive controller under adversarial attack.
\end{enumerate}
In Section \ref{sec:numexp}, we present numerical experiments to support our theoretical findings. Finally, concluding remarks are provided in Section \ref{sec:conclusion}.

\textbf{Notation.} Let $\mathbb{R}^n$ denote the set of $n$-dimensional vectors and $\mathbb{R}^{n\times n}$ denote the set of $n\times n$ matrices. For a matrix $A$, $\|A\|_2$ denotes the spectral norm and $\|A\|_F$ denotes the Frobenius norm of the matrix. 
Let $A^T$ denote the transpose of the matrix {and $\operatorname{Im}(A)$ denotes the column space of the matrix}. Let $I_n$ denote the identity matrix in $\mathbb{R}^{n\times n}$.
For matrices $A$ and $B$, $\langle A, B  \rangle$ denotes the Frobenius inner product.
$\|x\|_2$ denotes the $\ell_2$-norm of the vector. 
Let $\mathbb{E}$ denote the expectation operator.  
For an event $\mathcal{E}$, $\mathbb{P}(\mathcal{E})$ denotes the probability of the event, and the function $\mathbb{I}\{\mathcal{E}\}$ equals 1 if the event $\mathcal{E}$ occurs and 0 otherwise. The function
$\mathbb{I}_{\pm}\{E\}$ equals $1$ if $E$ occurs and $-1$ otherwise.
{Given functions $u$ and $v$, $u(x) = \mathcal{O}(v(x))$ implies $u(x) \le c_1 v(x)$ for some $c_1>0$,  $u(x) = \Omega(v(x))$ implies $u(x) \ge c_2 v(x)$ for some $c_2>0$, and $u(x) = \Theta(v(x))$ implies $c_3 v(x) \ge u(x) \ge c_4 v(x)$ for some $c_3\ge c_4>0$.}
Let $N(\mu,\Sigma)$ denote the Gaussian distribution with mean $\mu$ and covariance $\Sigma$. 
Finally, for a sub-Gaussian scalar variable $X$, let $\|X\|_{\psi_2}:= \inf\{t>0: \mathbb{E}[e^{X^2/t^2}]\leq 2\}.$

\section{Problem Formulation}\label{sec:probform}

Consider a linear time-invariant  system represented by:
\begin{equation}\label{sysd}
\begin{split}
    x_{t+1} &= A x_t + Bu_t+w_t,\\y_t&= Cx_t + Du_t, \hspace{10mm}t=0,1,\dots,
    \end{split}
\end{equation}
where $A\in\mathbb{R}^{n \times n}, B\in\mathbb{R}^{n \times m}, C\in\mathbb{R}^{r \times n}, D\in\mathbb{R}^{r \times m}$ are unknown system matrices, $x_t\in\mathbb{R}^n$ is the state, $u_t\in \mathbb{R}^m$ is the control input, and $y_t \in \mathbb{R}^r$ is the observation at time $t$.
$w_t\in \mathbb{R}^n$ is the attack injected into the system at time $t$ which occasionally takes a nonzero value. We assume that the attack times are selected with probability $p$, and 
$w_t$ is identically zero when the system is not under attack. We allow $w_t$ to be {adversarially} chosen at the attack times. 

Our goal is to identify the behavior of the system transferred from the control inputs $u_t, u_{t-1}, \dots, u_0$ to the output observation $y_t$, under partial observability. The mapping from a series of inputs to the observation is represented by the transfer function $C(zI-A)^{-1}B+D$, a function involving the coefficients $CB$, $CAB$, $CA^2 B$, and so forth. To this end, the following notion of a Markov parameter matrix represents the mapping of order $k$, which establishes the relationship from $u_t, \dots, u_{t-k+1}$ (the most recent $k$ inputs) to $y_t$.

\begin{definition}[Markov parameter matrix of order $k$]\label{markov}
 ~ From the true system $(A,B,C,D)$, the Markov parameter matrix required to recover the input-output mapping is denoted as $G^*$ and defined by 
    \begin{equation}\label{gstar}
        G^* = [D~CB~ CAB~ \cdots ~ CA^{k-2}B].
    \end{equation}
    Note that the size of $G^*$ depends on $k$; however, we omit the explicit dependence on $k$ in the notation for simplicity.
\end{definition}

By designing the control inputs $u_0, u_1, \dots$ and given the associated observation trajectory $y_0, y_1,\dots$,
our goal is to accurately approximate the input–output mapping $G^*$; to identify a Markov parameter matrix of order $k$. To this end, {for each $s\in\{1,\dots,k-1\}$,
consider the following relationship from $u_t,\dots, u_{t-k+1}$ to $y_t$ derived from \eqref{sysd}:
    \begin{subequations}\label{relationship}
         \begin{align}
    y_t &= G^*\cdot [u_t^T~ u_{t-1}^T ~ \cdots~ u_{t-k+1}^T]^T \numberthis\label{firstline}\\ &\hspace{4mm}+ [C~ CA ~ \cdots ~ CA^{s-1}] \cdot[w_{t-1}^T~ \cdots ~ w_{t-s}^T]^T\numberthis\label{secondline}\\&\hspace{4mm}+ [CA^{s}~  \cdots ~ CA^{k-2}] \cdot[w_{t-s-1}^T~ \cdots ~ w_{t-k+1}^T]^T\numberthis\label{thirdline}\\&\hspace{4mm}+ CA^{k-1}x_{t-k+1}, \numberthis\label{fourthline}
    \end{align}
    \end{subequations}
    where the term in \eqref{thirdline} is taken to be zero when $s = k-1$.
For the sake of convenience, we define \eqref{secondline} as
\begin{align}\label{defv}
    v_t^{(s)} := [C~ CA ~ \cdots ~ CA^{s-1}] \cdot[w_{t-1}^T~ \cdots ~ w_{t-s}^T]^T.
\end{align}}
Now, consider the following convex optimization problem with the $\ell_2$-norm estimator, given $T$ observation samples $y_{k-1},\dots, y_{T+k-2}$ and the control inputs $u_0, \dots, u_{T+k-2}$:
   \begin{equation}\label{l2est}
      \min_{G\in \mathbb{R}^{r\times mk}} \sum_{t=k-1}^{T+k-2}\|y_t-G\mathbf{U}_t^{(k)}\|_2,
   \end{equation}
   where $\mathbf{U}_t^{(k)} = [u_t^T ~ u_{t-1}^T ~ \cdots ~ u_{t-k+1}^T]^T$. Let $\hat G^*_T$ denote a minimizer of \eqref{l2est}.  

We will now present the assumptions required for an accurate estimation. To ensure that the estimation error $\|G^* - \hat G^*_T\|_F$ is sufficiently small, we impose the following conditions to prevent unbounded growth of the system states.

\begin{assumption}[System Stability]\label{stability}
There exist $\psi>0$ and $0\leq \rho<1$ such that 
$\|A^t\|_2\leq \psi\cdot \rho^t$ for all $t\geq 0$. 
\end{assumption}

\begin{assumption}[Sub-Gaussian norm]\label{maxnorm}
 Define a filtration {$\mathcal{F}_t = \bm{\sigma}\{x_0, x_1, \dots, x_t\}$.}
    There exists $\eta >0$ such that
  {for $x_0 = (x_{0,i})_{i=1}^n$ and $w_t = (w_{t,i})_{i=1}^n$, each coordinate satisfies $\|x_{0,i}\|_{\psi_2} \leq \eta$ and $\|w_{t,i}\;|\;\mathcal{F}_t\|_{\psi_2} \leq \eta$ for all $i \in \{1, \dots, n\}$ and $t \geq 0$.}
\end{assumption}

Note that Assumption \ref{stability} incorporates the standard system stability condition $\rho(A) < 1$, where $\rho(A)$ is the spectral radius of $A$. This follows from Gelfand's formula, $\Phi(A)=\sup_{t\geq 0} \frac{\|A^t\|_2}{\rho(A)^{t}}$, whose upper bound only depends on the system order $n$. Assumption \ref{maxnorm} requires each attack to be sub-Gaussian; this is automatically satisfied for any bounded attacks, which is reasonable since real-world actuators cannot take arbitrarily large inputs. We direct readers to Section II in \cite{kim2025system} to review the preliminaries regarding sub-Gaussian variables. 

The next assumption provides a criterion that enables us to identify the Markov parameter matrix of order $k$. 

\begin{assumption}[Probabilistic attack]\label{null}
 $w_t$ is an attack at each time $t$ with probability {$p<0.5$} conditioned on $\mathcal{F}_t$; i.e., there exists a sequence $(\xi_t)_{t\geq 0}$ of independent $\text{Bernoulli}(p)$ variables such that 
    $\{\xi_t = 0\} \subseteq \{w_t=0\}$ for all $t\geq 0.$
   {Under this restriction on attack times, $w_t$ can be adversarially chosen by an attacker with access to $\mathcal{F}_t$ at every attack time $t\ge 0$.}
\end{assumption}

Assumption \ref{null} implies that the system is not under attack ($w_t=0$) at time $t$ if $\xi_t=0$. 
Meanwhile, given $p<0.5$, select $s$ to be the largest integer in $\{1,\dots, k-1\}$ such that $p<\frac{1}{2s}$.  Then,
we can additionally infer a lower bound on the probability that $v_t^{(s)}=0$, since $v_t^{(s)}$ is exactly zero when $w_{t-1}, \dots, w_{t-s}$ are  zero (see \eqref{defv}). In other words, we have \begin{equation}\label{vexceed}\begin{split}\mathbb{P}(v_t^{(s)} &= 0) \geq \mathbb{P}(w_{t-1} = 0 , \dots, w_{t-s} = 0) \\&\geq \mathbb{P}(\xi_{t-1} = 0 , \dots, \xi_{t-s} = 0) \geq (1-p)^{s} > 0.5,\end{split}\end{equation}
where the second and third inequalities follow from the relationship between $w_t$ and $\xi_t$ as well as the independence of $\xi_t$'s, and the final inequality holds
under $p<\frac{1}{2s}$. In Theorem \ref{theorem1}, the inequality $1-(1-p)^{s}<0.5$ will play a crucial role in the accurate identification of the Markov parameter matrix.

\section{Batch Estimation: Identification with the $\ell_2$-norm estimator}\label{sec:l2est}

In this section, we provide a non-asymptotic analysis of the $\ell_2$-norm estimator using the \textit{batch data} of trajectory length $T$. {In the following subsections, we consider two attack schemes: first, a full-dimensional attack in $w_t\in\mathbb{R}^n$ as given in \eqref{sysd}, and second, an attack restricted to a lower-dimensional subspace via a linear filtering matrix. Notably, the first scheme represents the unrestricted attack baseline where the filtering matrix equals the identity $I_n$.}

\subsection{Full-Dimensional Attack Scheme}\label{sec:unconstrained}
{In this subsection, we  consider a full-dimensional attack scheme given in \eqref{sysd}.}
Previously in \cite{kim2025system}, we proved that {under $p<\frac{1}{2(k-1)}$}, the $\ell_1$-norm estimator achieves an error of {$\mathcal{O}(\sqrt{r}\cdot \rho^{k-1})$,} which exponentially decreases with $k$, after a finite time $T= \Omega(k^2 m+ k \log(\frac{r}{\delta}))$ with probability at least $1-\delta$. In this paper, we restate the claim with the $\ell_2$-norm estimator, and will prove that the recovery time can be slightly reduced by removing the logarithmic dependence on $r$; i.e., $T=\Omega(k^2 m+ k \log(\frac{1}{\delta}))$. Moreover, the $\ell_2$-norm estimator provides a better final estimation error than the $\ell_1$-norm estimator since we mitigate the multiplying factor $\sqrt{r}$ and achieve an error of {$\mathcal{O}(\rho^{k-1})$ under $p<\frac{1}{2(k-1)}$. More broadly, the analysis extends to any attack probability $p<0.5$: the error decays as $\mathcal{O}(\rho^s)$, where $s\in\{1,\dots,k-1\}$ is the largest integer satisfying $p<\frac{1}{2s}$, inducing a stepwise exponential decay in the estimation error as $p$ decreases.} We highlight that these results imply that the performance of the proposed $\ell_2$-norm estimator does not rely on the dimension of observations $r$.
To the best of our knowledge, this is the first result in the literature where the recovery time or the estimation error does not scale with the dimension of observations $r$, whereas existing works \cite{sarkar2021finite, oymak2022revisit} establish a dependence on $r$.

We first introduce some useful lemmas that leverage the Gaussianity of control inputs and sub-Gaussianity of attacks. 

\begin{lemma}[Gaussian Concentration Lemma]\label{gauscon}
    For an $M$-dimensional random vector $\mathbf{U}\sim N(\mu, \sigma^2 I_M)$ and the function $f$ having a Lipschitz constant of $L$, we have 
    \begin{align*}
        \|f(\mathbf{U}) - \mathbb{E}[f(\mathbf{U})]\|_{\psi_2} =\mathcal{O}(\sigma L). 
    \end{align*}
\end{lemma}

\begin{lemma}\label{sumxt}
    Suppose that Assumptions \ref{stability} and \ref{maxnorm} hold. Let the control inputs $\{u_t\}_{t=0}^{T+k-2}$ be independent with $u_t\sim N(0, \sigma^2 I_m)$. Then, we have
   \begin{align}\label{expxt}
       \mathbb{E}[\|x_t\|_2^2] =\mathcal{O}\biggr(\Bigr(\frac{ \eta{\sqrt{n}}+\sigma \sqrt{m} \|B\|_2}{1-\rho} \Bigr)^2\biggr), \quad\forall t\geq 0.
   \end{align}
   Moreover, given $\delta\in(0,1]$, when $T= \Omega(\log(\frac{1}{\delta }))$,
    \begin{equation}\label{23}
        \sum_{t=0}^{T-1} \|x_t\|_2 =\mathcal{O}\Bigr(\frac{(\eta{\sqrt{n}}+\sigma \sqrt{m} \|B\|_2)T}{1-\rho}\Bigr)
    \end{equation}
    holds with probability at least $1-\delta$. 
\end{lemma}
\begin{proof}
 The proofs of Lemmas \ref{gauscon} and \ref{sumxt} are provided in Appendix \ref{prooflemma1} and \ref{prooflemma2}, respectively.
\end{proof}

\medskip 

We now state the next result that is crucial for understanding that the recovery time of the $\ell_2$-norm estimator does not depend on the dimension of the observations. In other words, taking multiple observations at each time (thereby enlarging a Markov parameter matrix) does not worsen the recovery time since it is essentially equivalent to considering a single observation at each time. To this end, the following lemma shows that the worst-case recovery of the Markov parameter matrix occurs when the difference between the true and estimated matrices is rank-1.

{
\begin{lemma}\label{minrank1}
For $z\in\mathbb{R}^{mk}$, define
\[
h_T(z) := \sum_{t=k-1}^{T+k-2} \mathbb{I}_\pm \{v_t^{(s)}=0\} |z^T \mathbf{U}_t^{(k)}|, ~~ m_T := \min_{\|z\|_2=1} h_T(z).
\]
Suppose that $m_T > 0$. Then, every global minimizer of $\sum_{t=k-1}^{T+k-2} \mathbb{I}_\pm \{v_t^{(s)}=0\} \|Z \mathbf{U}_t^{(k)}\|_2$ over $Z\in\mathbb{R}^{r\times mk}$ satisfying $\|Z\|_F = 1$ has rank-1. Moreover, it holds that
\begin{align}\label{mTlowerbound}
\min_{\|Z\|_F=1}\mathbb{I}_\pm \{v_t^{(s)}=0\} \|Z \mathbf{U}_t^{(k)}\|_2  = m_T.
\end{align}
\end{lemma}
\begin{proof}
    The proof details are given in Appendix \ref{minrank1proof}.
\end{proof}
}

\medskip 

Finally, we present the main theorem of this section, which shows that after finite time, the $\ell_2$-norm estimator achieves an estimation error that {decays exponentially with $s\in\{1,\dots, k-1\}$, the  integer dictated by $p$ via $p<\frac{1}{2s}$. This implies that reducing the attack probability $p$ yields a stepwise exponential decrease in the estimation error bound. In the proof, we will show that $m_T>0$ holds with high probability, which satisfies the premise of Lemma \ref{minrank1}. This guarantees that \eqref{mTlowerbound} holds with high probability, which we then apply by setting  $Z = \frac{G^*-\hat G_T^*}{\|G^*-\hat G_T^*\|_F}$, the  normalized estimation error.}

\begin{theorem}\label{theorem1}
    Suppose that Assumptions \ref{stability}, \ref{maxnorm}, and \ref{null} hold. {Define $s$  as the largest integer in $\{1,\dots, k-1\}$ such that $p<\frac{1}{2s}$, and let $q_s:= 1-(1-p)^{s}$.}
      Further, let the control inputs $\{u_t\}_{t=0}^{T+k-2}$ be independent with $u_t\sim N(0, \sigma^2 I_m)$.
    Given $\delta\in(0,1]$, when 
    \begin{align}\label{mintime}
        T=\Omega\left(\frac{k}{(1-2{q_s})^2}\left[mk\log\Bigr(\frac{mk}{1-2{q_s}}\Bigr)+\log\Bigr(\frac{1}{\delta}\Bigr)\right]\right),
    \end{align}
    the estimation error of the Markov matrix is given as
    \begin{align}\label{esterror}
        \|G^* - \hat{G}^*_T\|_F =\mathcal{O}\biggr(\frac{\rho^{s}{\nu_s} }{1-\rho}\biggr)
    \end{align}
    with probability at least $1-\delta$, where $\hat G^*_T$ is a minimizer to \eqref{l2est} and ${\nu_s} = \frac{\|C\|_2}{1-2{q_s}}\cdot \bigr(\frac{\eta{\sqrt{n}}}{\sigma}+\sqrt{m}\|B\|_2\bigr)$.
\end{theorem}

\begin{proof}
   Given the finite time \eqref{mintime}, we will first derive a lower bound on 
    \begin{align}\label{gtrthan0}
        \sum_{t=k-1}^{T+k-2}\|(G^*-G)\mathbf{U}_t^{(k)} + {v_t^{(s)}}\|_2 -\| {v_t^{(s)}}\|_2.   
    \end{align}
    for any $G\in \mathbb{R}^{r\times mk}$. Note that we can write
    {\begin{align*}
        &\sum_{t=k-1}^{T+k-2}\|(G^*-G)\mathbf{U}_t^{(k)} + {v_t^{(s)}}\|_2 -\| {v_t^{(s)}}\|_2\\&\hspace{2mm}= \sum_{\substack{t=k-1,\\{v_t^{(s)}} = 0}}^{T+k-2}\|(G^*-G)\mathbf{U}_t^{(k)}\|_2 \\&\hspace{22mm}+ \sum_{\substack{t=k-1,\\{v_t^{(s)}} \neq 0}}^{T+k-2}\|(G^*-G)\mathbf{U}_t^{(k)} + {v_t^{(s)}}\|_2 -\|{v_t^{(s)}}\|_2 \\&\hspace{2mm}\geq \sum_{\substack{t=k-1,\\{v_t^{(s)}} = 0}}^{T+k-2}\|(G^*-G)\mathbf{U}_t^{(k)}\|_2 - \sum_{\substack{t=k-1,\\{v_t^{(s)}} \neq 0}}^{T+k-2}\|(G^*-G)\mathbf{U}_t^{(k)}\|_2\\&\hspace{2mm}= \sum_{t=k-1}^{T+k-2} \mathbb{I}_{\pm} \{ {v_t^{(s)}} = 0\} \cdot \|(G^*-G)\mathbf{U}_t^{(k)}\|_2 \\&\hspace{2mm}= \|G^*-G\|_F  \sum_{t=k-1}^{T+k-2} \mathbb{I}_{\pm} \{ {v_t^{(s)}} = 0\}  \Bigr\|\frac{G^*-G}{\|G^*-G\|_F}\cdot\mathbf{U}_t^{(k)}\Bigr\|_2,\numberthis\label{ultlowerbound}
    \end{align*}}
    where the first inequality follows from the triangle inequality and the last equality is from the positive homogeneity of the $\ell_2$-norm.

  {Since the Frobenius norm of $\frac{G^*-G}{\|G^*-G\|_F}$ is 1, obtaining a lower bound on \eqref{ultlowerbound} reduces to identifying a universal lower bound on $\sum_{t=k-1}^{T+k-2} \mathbb{I}_{\pm} \{ {v_t^{(s)}} = 0\} \cdot \|Z\mathbf{U}_t^{(k)}\|_2$ over $\|Z\|_F=1$. 
 Moreover, by Lemma \ref{minrank1}, we can further restrict $Z$ to rank-1 matrices. It therefore suffices to establish a universal positive lower bound on 
    \begin{align}\label{simple}
       f_z(\mathbf{U}) :=\sum_{t=k-1}^{T+k-2} \mathbb{I}_{\pm} \{ {v_t^{(s)}} = 0\} \cdot |z^T\mathbf{U}_t^{(k)}|,
    \end{align}
 defined for $\mathbf{U} = [u_0^T ~ u_1^T ~\cdots u_{T+k-2}^T]^T$ and $z\in\mathbb{R}^{mk}$ such that $\|z\|_2=1$.   
}

To obtain the Lipschitz constant of $f_{z}$, consider $\mathbf{U}, \mathbf{U}'\in \mathbb{R}^{m(T+k-1)}$ with   $\mathbf{U}' = [(u_0')^T ~ (u_1')^T ~\cdots (u_{T+k-2}')^T]^T$. We have
    \begin{align*}
     &  { |f_{z}(\mathbf{U}) - f_{z}(\mathbf{U}')|} \\&{\leq \sum_{t=k-1}^{T+k-2}
        | {z^T}[(u_t-u_t')^T ~ \dots ~ (u_{t-k+1}-u_{t-k+1}')^T]^T|} \\ &{\leq {\|z\|_2} \sum_{t=k-1}^{T+k-2}
        \|[(u_t-u_t')^T ~ \dots ~ (u_{t-k+1}-u_{t-k+1}')^T]\|_2 }\\&{\leq  \sqrt{T} \sqrt{\sum_{t=k-1}^{T+k-2}
        \|[(u_t-u_t')^T ~ \dots ~ (u_{t-k+1}-u_{t-k+1}')^T]\|_2^2 }}\\&{\leq  \sqrt{T} \sqrt{k}\sqrt{\sum_{t=0}^{T+k-2}\|u_t - u_t'\|_2^2} =  \sqrt{Tk}\|\mathbf{U} - \mathbf{U}'\|_2. }\numberthis\label{Tk+1}
    \end{align*}
    Thus, the Lipschitz constant of $f_{z}(\mathbf{U})$ is $\sqrt{Tk}$. Notice that $\mathbf{U}$ is a Gaussian with a variance of $\sigma^2 I_{m (T+k-1)}$. Then, it follows from Lemma \ref{gauscon} and the Hoeffding's inequality that 
    \begin{align*}
        \mathbb{P}\biggr(f_{z}(\mathbf{U}) - \mathbb{E}&\left[f_{z}(\mathbf{U})\right] \geq -\frac{T(1-2{q_s})\sigma}{\sqrt{2\pi}}\biggr) \\ &\geq 
        1-\exp\left(-\Omega\left(\frac{T (1-2{q_s})^2 }{ k}\right)\right).\numberthis\label{hoeffdingapply}
    \end{align*}
    Note that the Hoeffding's inequality is valid {since $1-2q_s>0$ holds under $p<\frac{1}{2s}$.}
Now, we lower bound $\mathbb{E}[f_{z}(\mathbf{U})]$.     Since the event $\{{v_t^{(s)}}=0\}$ includes the event $\{w_{t-1}=0, \dots, {w_{t-s}=0}\}$ (see \eqref{secondline}), we have
    \begin{align*}
        &\mathbb{E}[f_{z}(\mathbf{U})]=\mathbb{E}\left[ \sum_{t=k-1}^{T+k-2} \mathbb{I}_{\pm} \{ {v_t^{(s)}} = 0\} \cdot |{z^T}\mathbf{U}_t^{(k)}|\right]\\ &\hspace{3mm}\geq  T~ \mathbb{E}\left[\mathbb{I}_{\pm}\{w_{t-1}=0, \dots, {w_{t-s}=0}\} \cdot |{z^T}\mathbf{U}_t^{(k)}|\right] \\&\hspace{3mm}\geq 
        T~ \mathbb{E}\bigr[\mathbb{I}_{\pm}\{\xi_{t-1}=0, \dots, {\xi_{t-s}=0}\} \bigr] \cdot\mathbb{E}\bigr[ |{z^T}\mathbf{U}_t^{(k)}|\bigr] \\&\hspace{3mm}= T(1\cdot (1-{q_s})+(-1)\cdot {q_s})\cdot \sigma \sqrt{\frac{2}{\pi}}=T(1-2{q_s})\sigma \sqrt{\frac{2}{\pi}},\numberthis\label{fz}
    \end{align*}
    where the second inequality follows from the independence between $\mathbf{U}_t^{(k)}$ and  $\xi_{t-1}, \dots {\xi_{t-s}}$; the last inequality uses 
{the definition of the probability $q_s:=1-(1-p)^{s}$}. 

    Thus, considering \eqref{hoeffdingapply} and \eqref{fz},  when \begin{align}\label{onetime}
    T=\Omega\left(\frac{k}{(1-2{q_s})^2}\log\left(\frac{1}{\delta}\right)\right),\end{align} we have  \begin{align}\label{fixedz}
    f_{z}(\mathbf{U})\geq \frac{T(1-2{q_s}) \sigma }{\sqrt{2\pi}}
    \end{align}
    with probability at least $1-\delta$, for a {fixed $z\in\mathbb{R}^{mk}$ such that $\|z\|_2=1$}.

 We will now prove that $f_{z}(\mathbf{U})$ is bounded below by a positive constant for all {vectors $z$} with high probability. 
     Note that we have $f_z(\mathbf{U}) - f_{z'}(\mathbf{U}) \geq -\|z-z'\|_2 \sum_{t=k-1}^{T+k-2}  \|\mathbf{U}_t^{(k)}\|_2$ due to the triangle inequality. Also, $\sum_{t=k-1}^{T+k-2} \|\mathbf{U}_t^{(k)}\|_2$ has the  Lipschitz constant $\sqrt{Tk}$ similar to \eqref{Tk+1}, which incurs that 
     \begin{align*}
         &\mathbb{P}\left(\sum_{t=k-1}^{T+k-2} \|\mathbf{U}_t^{(k)}\|_2 - \mathbb{E}\biggr[\sum_{t=k-1}^{T+k-2} \|\mathbf{U}_t^{(k)}\|_2\biggr] \leq T \sigma \sqrt{mk}\right)\\&\hspace{5mm}\geq 1-\exp\left(-\Omega\left(\frac{ T^2 \sigma^2 mk}{\sigma^2 Tk}\right)\right)=1-\exp\left(-\Omega\left( Tm\right)\right).
     \end{align*}
     This implies that if $T=\Omega\bigr( \frac{1}{m}\log\left(\frac{2}{\delta}\right)\bigr)$, then \begin{align}\label{utkupbound}\mathbb{P}\biggr(\sum_{t=k-1}^{T+k-2} \|\mathbf{U}_t^{(k)}\|_2\leq 2 T\sigma \sqrt{mk}\biggr)\geq 1-\frac{\delta}{2},\end{align} since we have $\mathbb{E}[ \|\mathbf{U}_t^{(k)}\|_2]\leq \sigma\sqrt{mk}$ for $\mathbf{U}_t^{(k)}\sim N(0, \sigma^2 I_{mk})$. 
     Let $\epsilon := \frac{1-2{q_s}}{4\sqrt{2\pi \cdot mk}}$. Then, with probability at least $1-\frac{\delta}{2}$,
     \begin{align}\label{utkupbound2}
         f_z(\mathbf{U}) - f_{z'}(\mathbf{U}) &\geq -\epsilon \cdot 2T\sigma \sqrt{mk}  = -\frac{T(1-2{q_s})\sigma}{2\sqrt{2\pi}}
     \end{align}
     holds for all $z,z'$ such that $\|z-z'\|_2\leq \epsilon$. 
     This implies that if $f_{z'}(\mathbf{U}) \geq \frac{T(1-2{q_s})\sigma}{\sqrt{2\pi}}$,  all  points in $\{z: \|z-z'\|_2\leq \epsilon\}$ satisfy $f_{z}(\mathbf{U})\geq \frac{T(1-2{q_s})\sigma}{2\sqrt{2\pi}}$ with probability at least $1-\frac{\delta}{2}$.  
     By the covering number lemma \cite{vershynin2025high}, there exists an $\epsilon$-net $\mathcal{N}$ of at most $(1+\frac{2}{\epsilon})^{mk}$ points on the unit sphere such that $f_{z'}(\mathbf{U}) \geq \frac{T(1-2{q_s})\sigma}{\sqrt{2\pi}}$ holds concurrently for all $z' \in \mathcal{N}$ with probability at least $1-\frac{\delta}{2}$, obtained by replacing $\delta$ in \eqref{onetime} with $\delta / [2(1+\frac{2}{\epsilon})^{mk}]$ by a union bound. This requires
 \begin{align*}
         T&= \Omega\bigg(\frac{k}{(1-2{q_s})^2}\log\bigg(\frac{2(1+\frac{2}{\epsilon})^{mk}}{\delta}\bigg)\bigg) \\&= \Omega\bigg(\frac{k}{(1-2{q_s})^2 }\bigg[mk\log\Big(\frac{mk}{1-2{q_s}}\Big) + \log\Big(\frac{1}{\delta}\Big)\bigg]\bigg).
     \end{align*}
Combining this with a secondary union bound over \eqref{utkupbound2}, $f_{z}(\mathbf{U})\geq \frac{T(1-2{q_s})\sigma}{2\sqrt{2\pi}}$ holds for all unit vectors $z\in \mathbb{R}^{mk}$ with probability at least $1-\delta$ (note that the condition $T=\Omega( \frac{1}{m}\log(\frac{2}{\delta}))$ for \eqref{utkupbound2} is automatically satisfied).

     Thus, we have shown that after the finite time \eqref{mintime},  every {vector $z$} concurrently admits a positive lower bound on $f_{z}(\mathbf{U})$ with high probability. Due to Lemma \ref{minrank1}, $f_Z(\mathbf{U})$ is minimized at a rank-1 matrix $Z$ and $f_Z(\mathbf{U})$ has the same positive lower bound universally for all $Z$ such that $\|Z\|_F=1$. Recalling \eqref{ultlowerbound}, we ultimately obtain
     \begin{align*}
        & \sum_{t=k-1}^{T+k-2}\|(G^*-G)\mathbf{U}_t^{(k)} + {v_t^{(s)}}\|_2 -\| {v_t^{(s)}}\|_2 \\&\hspace{32mm}\geq \|G^* - G\|_F\cdot \frac{T(1-2{q_s})\sigma}{2\sqrt{2\pi}}\numberthis\label{lowerbo}
    \end{align*}
for any $G\in\mathbb{R}^{r\times mk}$ with high probability.
 
The $\ell_2$-norm estimator \eqref{l2est} is equivalent to 
    \begin{align}\label{l2estalter}
        \min_{G\in \mathbb{R}^{r\times mk}} &\sum_{t=k-1}^{T+k-2}\|(G^*-G)\mathbf{U}_t^{(k)} + {v_t^{(s)}}\\\nonumber&\hspace{8mm}{+\underbrace{CA^s \sum_{i=0}^{k-2-s} A^i w_{t-s-1-i}}_{\alpha_t}}+\underbrace{CA^{k-1} x_{t-k+1}}_{\beta_t}\|_2,
    \end{align}
 by considering \eqref{relationship} and \eqref{defv}.

For a minimizer $\hat{G}_T^*$ to \eqref{l2estalter}, the optimality incurs that
\begin{align*}
&\sum_{t=k-1}^{T+k-2} \|(G^*-\hat G^*_T)\mathbf{U}_t^{(k)} + {v_t^{(s)}}\|_2-\|\alpha_t\|_2-\|\beta_t\|_2 
\\&\leq\sum_{t=k-1}^{T+k-2} \|(G^*-\hat G^*_T)\mathbf{U}_t^{(k)} + {v_t^{(s)}} {+\alpha_t+\beta_t}\|_2 \\&\leq \sum_{t=k-1}^{T+k-2} \| {v_t^{(s)}} {+\alpha_t+\beta_t}\|_2 \leq \sum_{t=k-1}^{T+k-2} \| {v_t^{(s)}}\|_2 {+\|\alpha_t\|_2 + \|\beta_t\|_2},
\end{align*}
where the first and third inequalities are due to the triangle inequality. Rearranging the left- and right-hand sides, we have 
   \begin{align*}
         &\sum_{t=k-1}^{T+k-2}\|(G^*-\hat G^*_T)\mathbf{U}_t^{(k)} + {v_t^{(s)}}\|_2 -\| {v_t^{(s)}}\|_2 \\&\hspace{45mm}\leq {2\sum_{t=k-1}^{T+k-2} \|\alpha_t\|_2+\|\beta_t\|_2}, 
    \end{align*}
    Notice that after the finite time \eqref{mintime}, the inequality \eqref{lowerbo} provides a high-probability lower bound on the left-hand side with high probability. {For the right-hand side, Lemma \ref{sumxt} bounds the sum of  $\|\beta_t\|_2$ as $\sum_{t=k-1}^{T+k-2} \|\beta_t\|_2 = \sum_{t=0}^{T-1}\|CA^{k-1}x_t\|_2 = \mathcal{O}( \rho^{k-1}  \|C\|_2\frac{\eta\sqrt{n} + \sigma\sqrt{m}\|B\|_2}{1-\rho}T ) $ with high probability. Similarly, following  the same filtration argument as in the proof of Lemma \ref{sumxt}, the sum of $\|\alpha_t\|_2$  satisfies $\sum_{t=k-1}^{T+k-2}\|\alpha_t\|_2 \le\mathcal{O}( \rho^s\|C\|_2\frac{\sum_{t=0}^{T+k-4}\|w_t\|_2}{1-\rho})  = \mathcal{O}( \rho^s \|C\|_2\frac{\eta\sqrt{n}}{1-\rho}T)$ with high probability.}  Thus, we arrive at 
\begin{align*}
   & \|G^* - \hat G^*_T\|_F \cdot \frac{T(1-2{q_s})\sigma}{2\sqrt{2\pi} } \\&\hspace{25mm}= \mathcal{O}\Bigr({\rho^{s}}\|C\|_2 \frac{(\eta{\sqrt{n}}+\sigma \sqrt{m}\cdot \|B\|_2)T}{1-\rho}\Bigr) 
\end{align*}
with probability at least $1-\delta$,  which concludes that \eqref{esterror} holds after the finite time \eqref{mintime}.
\end{proof}

\begin{remark}\label{remark1}
    The final estimation error of the Markov parameter matrix with the $\ell_2$-norm estimator is {$\mathcal{O}(\rho^s)$}. 
    Note that this bound does not scale with  time $T$, meaning that long enough time  does not decay the error to zero. Instead, it {depends directly on $s$ determined by the attack probability $p<\frac{1}{2s}$}. Intuitively, any estimator would leverage the relationship \eqref{relationship}; for example, if the attacks $w_t$ are independent and zero-mean Gaussian for all $t$, prior works \cite{oymak2022revisit, sarkar2021finite} have shown that the least-squares method achieves $\mathcal{O}(1/\sqrt{T})$ error, but they leverage the fact that  the expectations of \eqref{secondline}--\eqref{fourthline} are zero. In our setting,  {$p<\frac{1}{2s}$} ensures that \eqref{secondline} equals zero with probability more than $0.5$, whereas the expectation or the geometric median of {\eqref{thirdline}--\eqref{fourthline}} is unknown since past adversarial  attacks are inherently incorporated into {these terms}; all we know is that the factor $A^{s}$ shrinks the error bound. Due to the lack of information under adversarial attacks, the best achievable bound on the estimation error is {$\mathcal{O}(\rho^s)$}.
\end{remark}

\begin{remark}
    To provide a clean estimation error bound,  we  prove Theorem \ref{theorem1} when each $u_t$ has zero mean. However, it is possible to extend this result for nonzero-mean Gaussian, {such as in closed-loop identification where an agent applies its own closed-loop control policy plus zero-mean Gaussian excitation.} For example, \eqref{fz} still holds for nonzero-mean $\mathbf{U}_t^{(k)}$ since shifting the mean for Gaussian variables always provides a greater expectation; i.e., $\mathbb{E}[\|Z\mathbf{U}_t^{(k)}\|_2] \leq \mathbb{E}[\|Z(\mathbf{U}_t^{(k)}+\alpha)\|_2]$ for any $\alpha\in \mathbb{R}^{mk}$. Moreover,  Lemma \ref{gauscon} holds regardless of the expectation of Gaussian variables. The only line of proof that does not hold is 
\eqref{utkupbound} since we cannot guarantee that $\mathbb{E}[\|\mathbf{U}_t^{(k)}\|_2] \le \sigma \sqrt{mk}$ for the nonzero-mean case; however, imposing an extra bound on the expectation of $u_t$ will easily resolve the issue,  with some additional constants appearing in the error bound. 
{Therefore, although mitigating the zero-mean control input assumption is crucial for practical closed-loop control, analyzing the zero-mean case is sufficient to establish the primary bounds, with generalization to any feedback policies being straightforward (see also \cite{kim2026sharp} for extensions beyond Gaussian excitation).}
\end{remark}

\subsection{Low-Dimensional Subspace Attack Scheme}\label{sec:subspace}

{In this subsection, we consider the subspace-constrained attack model given as:
\begin{equation}\label{sysd_weakened}
\begin{split}
    x_{t+1} &= A x_t + Bu_t+Ew_t,\\y_t&= Cx_t + Du_t, \hspace{10mm}t=0,1,\dots,
    \end{split}
\end{equation}
where $E\in \mathbb{R}^{n\times e}$ is an attack filter, $w_t\in \mathbb{R}^e$ is the adversarial attack, and the resulting filtered attack $Ew_t$ is restricted to the subspace $\operatorname{Im}(E)$. 
We consider the case where the observation matrix $C$ and the attack filter $E$ are known. 
Nevertheless, the attacks $w_t$ remain unknown and are  adversarially chosen according to  Assumption \ref{null}, while Assumption \ref{maxnorm} is naturally adapted from $\mathbb{R}^n$ to $\mathbb{R}^e$.

Now, we impose the following rank condition:
\begin{assumption}[Full column rank]\label{as:fullcolumn1}
   Suppose that $C$ and $E$ are known and
      $\operatorname{rank}(CE) = e.$
\end{assumption}
 A necessary condition for Assumption~\ref{as:fullcolumn1} is $r \ge e$, meaning that the attack dimension cannot exceed the  observation dimension.

In Section \ref{sec:unconstrained},  the estimation error bound degrades to $\mathcal{O}(\rho)$ when $p$ is close to $0.5$, which may be overly conservative. In this subsection, we overcome this limitation and show that an error bound of $\mathcal{O}((\frac{\rho}{1-p})^{k-1})$ is achievable regardless of $p$ (as long as $p<0.5$), given a sufficiently long time horizon.  

Instead of the standard $\ell_2$-norm estimator \eqref{l2est}, we solve the following alternative constrained $\ell_2$-norm estimator given as:
 \begin{equation}\label{l2est_alt}
      \min_{H\in\mathbb{R}^{e\times r},J\in\mathbb{R}^{e\times mk} } \sum_{t=k-1}^{T+k-2}\|Hy_t-J\mathbf{U}_t^{(k)}\|_2 \quad \text{s.t.}\quad HCE= I_e,
   \end{equation}

As in the previous subsection, we establish an estimation error bound for each $s\in\{1,\dots, k-1\}$. To this end, we add the following assumption:
\begin{assumption}[Subspace separability]\label{as:separability}
  Suppose that 
     \begin{align*}
         \operatorname{Im}(CE)\cap
\operatorname{Im}([CAE\ \dots\ CA^{s-1}E])
=\{0\}.
     \end{align*} 
\end{assumption}
Assumption \ref{as:fullcolumn1} implies that $CE$ preserves all $e$ independent directions, while Assumption \ref{as:separability} implies that none of these directions can be represented as a linear combination of the directions generated by $CAE,\dots,CA^{s-1}E$.

\begin{theorem}
    Suppose that Assumptions \ref{stability}, \ref{maxnorm}, \ref{null}, \ref{as:fullcolumn1}, and \ref{as:separability} hold. Let the control inputs $\{u_t\}_{t=0}^{T+k-2}$ be independent with $u_t \sim \mathcal{N}(0, \sigma^2 I_m)$. Consider a minimizer $(\hat H, \hat J)$ to the constrained $\ell_2$-norm estimator \eqref{l2est_alt}. Given $\delta \in (0, 1]$, when
\begin{align}\label{final_timebound2}
    T = \tilde{\Omega}\bigg( \frac{e m k^2 s}{(1-p)^{s-1}} + \frac{\eta^2 e\cdot (e m k s)^2}{\rho^{2s} } \bigg),
\end{align}
where $\tilde \Omega$ hides logarithmic factors, then the relative estimation error satisfies
\begin{align*}
\frac{\|\hat HG^*-\hat J\|_F}{\|\hat H\|_2} = \mathcal{O}\bigg(\Big(\frac{\rho}{1-p}\Big)^s\frac{(1-p)(1 + \frac{1}{\sin\theta}) \nu}{(1-\rho)(1-2p)} \bigg)  
\end{align*}
with probability at least $1-\delta$, where $\nu:= \sqrt{m} \|B\|_2\|C\|_2+\frac{1+\eta {\sqrt{n}}\|C\|_2+ \eta\sqrt{e}\|C\|_2\|E\|_2}{\sigma}$ and $\theta\in (0,\pi/2]$ denotes the minimal principal angle between $\operatorname{Im}(CE)$ and $\operatorname{Im}([CAE ~ \dots ~ CA^{s-1}E])$.
\end{theorem}

\textit{Proof Sketch. }
    Due to space constraints, the full proof is deferred to the online technical note \cite{kim2026defense}.
    In the proof, we show that under Assumptions \ref{as:fullcolumn1} and \ref{as:separability}, there exists a matrix $H^* \in \mathbb{R}^{e \times r}$ such that
\begin{align}\label{I_0}
H^* CE = I_e \quad \text{and} \quad H^* \cdot  [CAE ~ \dots ~ CA^{s-1}E] = 0.
\end{align}
We then analyze an adapted variant of the quantity in \eqref{gtrthan0}:
 \begin{align}\label{hj}
     \sum_{t=k-1}^{T+k-2} {\|(HG^*-J)\mathbf{U}_t^{(k)} + H \tilde v_t^{(s)}\|_2} - {\|H^* \tilde v_t^{(s)}\|_2}
 \end{align}
where $\tilde v_t^{(s)}$ is modified from \eqref{defv} as
 \begin{align*}
     \tilde v_t^{(s)}:=    [CE~ CAE ~ \cdots ~ CA^{s-1}E] \cdot[w_{t-1}^T~ \cdots ~ w_{t-s}^T]^T.
 \end{align*}
By construction of $H^*$,  we have $H^* \tilde v_t^{(s)} = w_{t-1}$. Moreover, under the constraint $HCE=I_e$ in \eqref{l2est_alt}, we can write $H\tilde v_t^{(s)}= w_{t-1} + H [ CAE ~ \cdots ~ CA^{s-1}E] \cdot[w_{t-2}^T~ \cdots ~ w_{t-s}^T]^T$. 
Substituting these relationships allows us to establish a uniform lower bound on \eqref{hj} by leveraging techniques used in the proof of Theorem \ref{theorem1}.
Modifications are required since the expression is not structured as in \eqref{ultlowerbound} and we cannot directly exploit the homogeneity of norm functions.
\hfill $\blacksquare$

The relative estimation error is $\mathcal{O}\big((\frac{\rho}{1-p})^s \cdot \frac{1-p}{(1-\rho)(1-2p)}\big)$, implying that the error exponentially decreases as $s$ increases, as long as the attack probability satisfies $p <\min\{0.5,  1-\rho\}$, a significant improvement from $p<\frac{1}{2s}$ in Theorem \ref{theorem1}. Note that 
we cannot directly derive an estimate of $G^*$ from $(\hat{H}, \hat{J})$ since the system $\hat{H}G^* \approx \hat{J}$ can be underdetermined. Instead, we can use $(\hat{H}, \hat{J})$ to determine whether an attack has happened at time $t$ by evaluating the residual $\hat{H} y_t - \hat{J} \mathbf{U}_t^{(k)}$. 

 \begin{remark}
In our previous derivation of \eqref{ultlowerbound}, establishing the lower bound has relied on partitioning data samples into attack-free ($v_t^{(s)} = 0$) versus corrupted ($v_t^{(s)} \neq 0$) cases and requiring that attack-free $s$-step windows outnumber corrupted ones. This necessitates the strict condition $p < \frac{1}{2s}$ (see \eqref{vexceed}).
On the other hand, the existence of $H^*$  enables partitioning based on  $w_{t-1} = 0$ versus $w_{t-1} \neq 0$; i.e., a single attack occurrence, where the condition $p < 0.5$ is sufficient to ensure that attack-free samples ultimately dominate corrupted ones.
This significant relaxation of attack probability  incurs a trade-off: when $w_{t-1} = 0$, the disturbance term $H \tilde v_t^{(s)}$ does not fully vanish unless 
entire $s$-step window is fully uncorrupted. This implies that a positive lower bound on \eqref{hj} is established by fully uncorrupted windows, since any positive contribution from partially corrupted windows cannot be quantified. Given $w_{t-1}=0$, such clean windows occur with probability at least $(1-p)^{s-1}$, affecting $T$ in \eqref{final_timebound2}. Thus, the burden shifts from restricting the admissible attack probability $p$ to increasing the required recovery time $T$.
\end{remark}
}

\section{Streaming Estimation: Stochastic Subgradient Descent Algorithm}\label{sec:subg}

In Section \ref{sec:unconstrained}, we showed that {for $p<\frac{1}{2s}$}, there exists a recovery time after which the $\ell_2$-norm estimator achieves an error of $\mathcal{O}(\rho^{s})$ using the batch data of length at least the time threshold. However, its behavior \textit{prior to the recovery time} is still unknown. To this end, we propose a subgradient descent algorithm based on the $\ell_2$-norm estimator and show that it works for the \textit{streaming data} while  providing a provable bound on the estimation error. {
This approach is more compelling for the low-dimensional attack scheme in Section \ref{sec:subspace} that relaxes the attack probability while requiring a longer recovery time, elevating the need for pre-recovery  guarantees. We leave the streaming extension for this low-dimensional scheme to future work and focus here on the general attack setting, specifically under $p<\frac{1}{2(k-1)}$ where $s=k-1$ in Theorem \ref{theorem1}.}

In contrast to smooth estimators such as least squares for which closed-form solutions exist, the $\ell_2$-norm estimator does not admit a closed-form expression, nor does it allow recursive updates of the estimate from time $t$ to time $t+1$. This motivates the use of the proposed subgradient descent algorithm.  

Meanwhile, computing the full subgradient at time $t$ indeed requires $\Theta(t)$ operations, and thus the total computational complexity grows quadratically with time if every iteration uses the full subgradient. Thus, to reduce complexity, the proposed algorithms adopt a stochastic approach, which only considers one observation trajectory sample at a time, to provide an unbiased approximation of the full gradient. 
Algorithms \ref{Algorithm 1} and \ref{Algorithm 2} respectively leverage the knowledge of the true Markov parameter matrix $G^*$; the former requires the inner product between the selected subgradient and $G^*$, and the latter (Polyak step size) requires $f_t(G^*)$ which involves the knowledge of \eqref{secondline} and \eqref{thirdline}.
Despite such challenges, we offer some flexibility to choose step sizes via the parameter $\alpha$. 
We first show that these algorithms show linear convergence until the algorithm reaches $\mathcal{O}(\rho^{k-1})$, the error achieved by the $\ell_2$-norm estimator using the batch data.

\begin{algorithm}[!t]
  \caption{Stochastic Subgradient Descent Algorithm with the Best\protect\footnotemark[1] step size}
  
  \begin{flushleft}
        \textbf{Input:} Time Horizon $T$. Initialization of the Markov parameter matrix $G^{(0)} = 0$.  Flexibility parameter $0<\alpha\leq 1$. {Independent control inputs $u_t \sim \mathcal{N}(0, \sigma^2 I_m)$ for all $t \ge 0$.}
  \end{flushleft}

  \begin{algorithmic}[1]
    \FOR{Time $t=0,~k,~2k,\dots$}
        \STATE Consider $f_t(G) = 
        \|y_{t+k-1} - G\mathbf{U}_{t+k-1}^{(k)}\|_2$.
        \STATE Select\footnotemark[2] $g^{(t)} \in \partial f_t(G^{(t)})$.
        
        \STATE Let $\theta_t = \frac{ \langle g^{(t)}, G^{(t)}-G^* \rangle}{\|g^{(t)}\|_F^2}.$ \alglinelabel{stepsize}

        \STATE Select the step size $\alpha \theta_t \leq  \gamma_t \leq (2-\alpha)\theta_t.$ 
        \STATE Let $G^{(t+k)} = G^{(t)} - \gamma_t g^{(t)}.$

    \ENDFOR 
  \end{algorithmic}

  \begin{flushleft}
        \textbf{Output:} Estimates at times that are multiples of $k$: $G^{(k)}, G^{(2k)}, G^{(3k)}, \dots$.  
  \end{flushleft}

  \label{Algorithm 1}
\end{algorithm}

\footnotetext[1]{We refer to this step size rule as the ``best'' since $\theta_t$ in Line \ref{stepsize} corresponds to the minimizer of the function in \eqref{optimal} with respect to $\gamma_t$.}

\footnotetext[2]{To ensure that Line \ref{stepsize} is valid, we note that the $\ell_2$-norm is non-smooth at points where its subdifferential contains the zero vector, and thus a nonzero subgradient $g^{(t)}$ can always be chosen.}


\begin{theorem}\label{linconv}
{Suppose that Assumptions \ref{stability}, \ref{maxnorm}, and  \ref{null} hold with $p<\frac{1}{2(k-1)}$, and  let $q:= 1-(1-p)^{k-1}$.
}
For all $t\geq 0$ that are multiples of $k$, let $G^{(t)}$ denote the estimate produced by Algorithm \ref{Algorithm 1} at time $t$.
Then, we have
{\small\begin{align*}
\nonumber &\mathbb{E}[\|G^{(t)}-G^*\|_F^2] \\&= \mathcal{O}\biggr(\Big(1-\frac{(1-2q)^2 (2\alpha-\alpha^2) }{2\pi\cdot mk}\Big)^{t/k} \|G^{(0)} - G^*\|_F^2+\Bigr(\frac{\rho^{k-1}\nu }{1-\rho} \Bigr)^2\biggr),  
\end{align*}}
where $\nu=\frac{\|C\|_2}{1-2q}\left(\frac{\eta{\sqrt{n}}}{\sigma}+\sqrt{m}\|B\|_2\right)$.
\end{theorem}

\begin{proof}
 Following the step size rule, we have 
 \begin{align*}
     &\|G^{(t+k)}-G^*\|_F^2  = \|G^{(t)}-\gamma_t g^{(t)}-G^*\|_F^2\\&= \|G^{(t)}-G^*\|_F^2 -2\gamma_t \langle g^{(t)}, G^{(t)}-G^*\rangle + \gamma_t^2 \|g^{(t)}\|_F^2  \numberthis\label{optimal}
     \\&= \|G^{(t)}-G^*\|_F^2  - (2\alpha-\alpha^2)\frac{\langle g^{(t)}, G^{(t)}-G^*\rangle ^2}{\|g^{(t)}\|_F^2}. \numberthis\label{noexp}
 \end{align*}

Now, for all $t\geq 0$ that are multiples of $k$,
define the filtration $\mathcal{G}_t = \bm{\sigma}\{G^{(0)},G^{(k)},G^{(2k)}, \dots, G^{(t)}\}$. 
Let $J_t$ be the $\mathcal{G}_t$-measurable event that 
    \begin{align}\label{xtcondition}
         \|G^{(t)}-G^*\|_F\geq   \frac{\sqrt{2\pi}}{(1-2q)\sigma}\mathbb{E}[\|CA^{k-1} x_{t}\|_2~|~\mathcal{G}_t]
    \end{align}
    holds. Also, let $J_t^c$ denote its complement. 
 We now decompose
\begin{align*}
\mathbb{E}[\|G^{(t+k)} -  G^*\|_F^2~|~\mathcal{G}_t] &= \underbrace{ \mathbb{E}\bigr[\mathbb{I}\{J_t^c\}\cdot\|G^{(t+k)} -  G^*\|_F^2~|~\mathcal{G}_t\bigr]}_{[\text{Term~} 1]} \\&\hspace{1mm}+ \underbrace{  \mathbb{E}\bigr[\mathbb{I}\{J_t\}\cdot\|G^{(t+k)} -  G^*\|_F^2~|~\mathcal{G}_t\bigr]}_{[\text{Term~} 2]},\numberthis\label{decomposition}
\end{align*}
and bound each term on the right-hand side separately.

Observe that \eqref{noexp} implies that the estimation error does not increase with the proposed step size; i.e., $\|G^{(t+k)}-G^*\|_F^2 \leq \|G^{(t)}-G^*\|_F^2$. Thus, we have
\begin{align}\label{term1}
    [\text{Term~}1] \leq \mathbb{I}(J_t^c)\cdot \|G^{(t)}-G^*\|_F^2.
\end{align}

\begin{algorithm}[!t]
  \caption{Algorithm \ref{Algorithm 1} with the Polyak step size}

\begin{flushleft}

    \textcolor{blue}{// Replace Line \ref{stepsize} of Algorithm \ref{Algorithm 1} with the following:}
\end{flushleft}
\begin{algorithmic}

  \STATE Let the step size $\theta_t = \max\{\frac{f_t(G^{(t)}) - f_t(G^*)}{\|g^{(t)}\|_F^2},0\}$.

\end{algorithmic}

  \label{Algorithm 2}
\end{algorithm}

Now, we focus on [Term $2$] and analyze the related term 
   $\mathbb{E}\bigr[ \mathbb{I}\{J_t\}\cdot (f_t(G^{(t)})-f_t(G^* ))~|~\mathcal{G}_t\bigr]$. We first recall \eqref{relationship} to yield 
\begin{align*}
   & \mathbb{E}[f_t(G^{(t)})-f_t(G^* )~|~\mathcal{G}_t]\\& = \mathbb{E}\bigr[ \|(G^{(t)}-G^*)\mathbf{U}_{t+k-1}^{(k)}+v_{t+k-1} +CA^{k-1} x_{t} \|_2\\&\hspace{45mm} - \|v_{t+k-1} +CA^{k-1} x_{t} \|_2 ~\bigr|~\mathcal{G}_t\bigr]\\&\geq
    \mathbb{E}\bigr[\mathbb{I}\{v_{t+k-1}=0\}\cdot(\|(G^{(t)}-G^*)\mathbf{U}_{t+k-1}^{(k)} +CA^{k-1} x_{t} \|_2 \\&\hspace{57mm} - \|CA^{k-1} x_{t} \|_2 )~\bigr|~ \mathcal{G}_t\bigr] \\&\hspace{5mm}+ 
     \mathbb{E}\bigr[- \mathbb{I}\{v_{t+k-1}\neq 0\} \cdot \|(G^{(t)}-G^*)\mathbf{U}_{t+k-1}^{(k)}\|_2  ~\bigr|~\mathcal{G}_t\bigr]\\&\geq (1-q)\mathbb{E}[\|(G^{(t)}-G^*)\mathbf{U}_{t+k-1}^{(k)}\|_2~|~\mathcal{G}_t]\\&\hspace{1mm}-(1-q)\mathbb{E}[\|CA^{k-1} x_{t}\|_2|\mathcal{G}_t]- q\mathbb{E}[\|(G^{(t)}-G^*)\mathbf{U}_{t+k-1}^{(k)}\|_2|\mathcal{G}_t]\\&\geq(1-2q)\cdot \|G^{(t)}-G^*\|_F \cdot  \sigma \sqrt{\frac{2}{\pi}}-\mathbb{E}[\|CA^{k-1}x_t\|_2~|~\mathcal{G}_t],
\end{align*}
where the first inequality is due to the triangle inequality in the case of $\{v_{t+k-1}\neq 0\}$, and the second inequality follows from  the fact that 
$\mathbb{E}[\|(G^{(t)}-G^*)\mathbf{U}_{t+k-1}^{(k)} +CA^{k-1} x_{t}\|_2 ~|~\mathcal{G}_t]\geq \mathbb{E}[\|(G^{(t)}-G^*)\mathbf{U}_{t+k-1}^{(k)}\| ~|~\mathcal{G}_t]$
due to 
the symmetry of a zero-mean Gaussian variable $\mathbf{U}_{t+k-1}^{(k)}=[u_{t+k-1}, \dots, u_{t}]$ being independent of  $\mathcal{G}_t$ and $x_t$. Note that the event $\{v_{t+k-1}=0\}$ including the event $\{\xi_{t+k-2}=0, \dots, \xi_t = 0\}$ applies to relate the expression with $q$, similar to the approach in \eqref{fz}.

Noting that \eqref{xtcondition} holds under $J_t$, we finally arrive at
\begin{align*}
   & \mathbb{E}\bigr[ \mathbb{I}\{J_t\}\cdot (f_t(G^{(t)})-f_t(G^* ))~|~\mathcal{G}_t\bigr]\\&\geq  \mathbb{E}\biggr[ \mathbb{I}\{J_t\}\cdot \Bigr((1-2q)\cdot \|G^{(t)}-G^*\|_F \cdot  \sigma \sqrt{\frac{2}{\pi}}\\&\hspace{30mm}-\mathbb{E}[\|CA^{k-1}x_t\|_2~|~\mathcal{G}_t]\Bigr)~\biggr|~\mathcal{G}_t\biggr]\\ &\geq\mathbb{I}\{J_t\}\cdot \Bigr( \|G^{(t)}-G^*\|_F\cdot (1-2q)  \sigma\sqrt{\frac{1}{2\pi}}\Bigr).\numberthis\label{lbft}
\end{align*}

We now take expectation of both sides of \eqref{noexp} conditioned on $\mathcal{G}_t$, which leads to
\begin{align*}
    &\mathbb{E}\bigr[\mathbb{I}\{J_t\}\cdot\|G^{(t+k)}-G^*\|_F^2~|~ \mathcal{G}_t\bigr] \\&= \mathbb{I}\{J_t\}\cdot\|G^{(t)}-G^*\|_F^2  \\&\hspace{4mm}- (2\alpha-\alpha^2) \mathbb{E}\left[ \frac{\mathbb{I}\{J_t\}^2\cdot\langle g^{(t)}, G^{(t)}-G^*\rangle ^2}{\|g^{(t)}\|_F^2}~\Biggr|~\mathcal{G}_t\right] \\&\leq \mathbb{I}\{J_t\}\cdot\|G^{(t)}-G^*\|_F^2 \\&\hspace{4mm} -  (2\alpha-\alpha^2)\frac{\mathbb{E}\bigr[\mathbb{I}\{J_t\}\cdot\langle g^{(t)}, G^{(t)}-G^*\rangle ~|~\mathcal{G}_t\bigr]^2}{\mathbb{E}[\|g^{(t)}\|_F^2~|~\mathcal{G}_t]} \\&\leq \mathbb{I}\{J_t\}\cdot\|G^{(t)}-G^*\|_F^2 \\&\hspace{4mm} -  (2\alpha-\alpha^2)\frac{\mathbb{E}\bigr[\mathbb{I}\{J_t\}\cdot  (f_t(G^{(t)})-f_t(G^* )) ~|~\mathcal{G}_t\bigr]^2}{\mathbb{E}[\|g^{(t)}\|_F^2~|~\mathcal{G}_t]},\numberthis\label{inequality1}
\end{align*}
where the equality uses $\mathbb{I}\{J_t\}^2 = \mathbb{I}\{J_t\}$ and the first inequality follows from the Cauchy-Schwarz inequality. 
The last inequality comes from
\begin{align*}
    \mathbb{E}\bigr[ \mathbb{I}\{J_t\} \cdot&\langle g^{(t)}, G^{(t)}-G^*\rangle ~|~\mathcal{G}_t \bigr]\\&\geq \mathbb{E}\bigr[ \mathbb{I}\{J_t\} \cdot(f_t(G^{(t)})-f_t(G^* )) ~|~\mathcal{G}_t \bigr]\geq 0,
\end{align*}
which follows from the convexity of $f_t$ and its subgradient inequality  $f_t(G^*) - f_t(G^{(t)})\geq\langle g^{(t)}, G^*-G^{(t)}\rangle$, while the nonnegativity is guaranteed by \eqref{lbft}. 
Moreover, we have that 
\begin{align*}
    \mathbb{E}[\|g^{(t)}\|_F^2~|~\mathcal{G}_t] &\leq \sup_{\|z\|_2\leq 1 } \mathbb{E}\left[\bigr\| z (\mathbf{U}_{t+k-1}^{(k)})^T\bigr\|_F^2
\right] \\&= \sup_{\|z\|_2\leq 1} \mathbb{E}\left[\text{trace}\bigr(z (\mathbf{U}_{t+k-1}^{(k)})^T\mathbf{U}_{t+k-1}^{(k)}z^T\bigr)\right]\\&\leq \sup_{\|z\|_2\leq 1 }\mathbb{E}\Bigr[(\mathbf{U}_{t+k-1}^{(k)})^T\mathbf{U}_{t+k-1}^{(k)}\cdot z^T z\Bigr] \\&\leq \mathbb{E}\Bigr[(\mathbf{U}_{t+k-1}^{(k)})^T\mathbf{U}_{t+k-1}^{(k)}\Bigr] 
\leq \sigma^2 mk,\numberthis\label{ubgtsq}
\end{align*}
where we used the cyclic property of trace. We apply \eqref{lbft} and \eqref{ubgtsq} to the inequality \eqref{inequality1} to obtain
\begin{align*}
    &[\text{Term~}2]=\mathbb{E}\bigr[\mathbb{I}\{J_t\}\cdot\|G^{(t+k)}-G^*\|_F^2~|~ \mathcal{G}_t\bigr] \\&\leq \mathbb{I}\{J_t\}\cdot \|G^{(t)}-G^*\|_F^2 \\&\hspace{10mm}- (2\alpha-\alpha^2)\frac{\mathbb{I}\{J_t\}^2\cdot \|G^{(t)}-G^*\|_F^2 \cdot (1-2q)^2 \sigma^2 }{2\pi \cdot \sigma^2 mk}\\&\leq\mathbb{I}\{J_t\}\cdot  \left(1-\frac{(1-2q)^2  (2\alpha-\alpha^2)}{2\pi \cdot mk}\right)\|G^{(t)}-G^*\|_F^2. \numberthis\label{term2}
\end{align*}

Now, we are ready to integrate $[\text{Term~}1]$ and $[\text{Term~}2]$ via \eqref{decomposition}.
For simplicity, define $\bar{c}:= \frac{(1-2q)^2  (2\alpha-\alpha^2)}{2\pi \cdot mk}$. Then, considering \eqref{term1} and \eqref{term2}, we arrive at 
\begin{align*}
&\mathbb{E}[\|G^{(t+k)}-G^*\|_F^2~|~\mathcal{G}_t] = [\text{Term~}1]+[\text{Term~}2]\\&\leq \mathbb{I}(J_t^c)\cdot \|G^{(t)}-G^*\|_F^2 + \left(1-\bar c\right)\cdot\mathbb{I}\{J_t\}  \|G^{(t)}-G^*\|_F^2 \\
&= ((1-\bar c)+\bar c)\cdot\mathbb{I}(J_t^c)  \|G^{(t)}-G^*\|_F^2 \\&\hspace{40mm}+ \left(1-\bar c\right)\cdot \mathbb{I}\{J_t\}  \|G^{(t)}-G^*\|_F^2
\\&= \left(1-\bar c \right)\|G^{(t)}-G^*\|_F^2+\bar c \cdot \mathbb{I}(J_t^c) \|G^{(t)}-G^*\|_F^2,\numberthis\label{barc}
\end{align*}
where we use $\mathbb{I}\{J_t^c\}+\mathbb{I}\{J_t\}=1$ for the last equality. 
Note that $J_t^c$ is the complement of \eqref{xtcondition}. Then, we have
\begingroup
\fontsize{9.5pt}{11.5pt}\selectfont
\begin{align*}
    \mathbb{I}(J_t^c) \|G^{(t)}-G^*\|_F^2&\leq  \mathbb{I}(J_t^c)\biggr(\frac{\sqrt{2\pi}}{(1-2q)\sigma}\mathbb{E}[\|CA^{k-1} x_{t}\|_2~|~\mathcal{G}_t]\biggr)^2\\&\leq \mathbb{I}(J_t^c)\biggr(\frac{\sqrt{2\pi}}{(1-2q)\sigma}\biggr)^2\mathbb{E}[\| CA^{k-1}x_{t}\|_2^2~|~\mathcal{G}_t] \\&\leq\biggr(\frac{\sqrt{2\pi}}{(1-2q)\sigma}\|C\|_2\|A^{k-1}\|_2\biggr)^2\mathbb{E}[\| x_{t}\|_2^2~|~\mathcal{G}_t],
\end{align*}
\endgroup%
where we use Jensen's inequality and  $\mathbb{I}(J_t^c)\leq 1$. Substituting this inequality into \eqref{barc} and recursively applying the tower rule with respect to $\mathcal{G}_{t-k}, \dots,\mathcal{G}_{0}$ finally yields
\begin{align*}
    &\mathbb{E}[\|G^{(t)}-G^*\|_F^2] \\&\leq 
   (1-\bar c)^{t/k} \|G^{(0)}-G^*\|_F^2+ \biggr(\frac{\sqrt{2\pi}}{(1-2q)\sigma}\|C\|_2\|A^{k-1}\|_2\biggr)^2\\&\hspace{43mm}\times  \bar c \sum_{i=1}^{t/k} (1-\bar c)^{i-1} \mathbb{E}[\|x_{t-ki}\|_2^2]\\&\leq (1-\bar c)^{t/k} \|G^{(0)}-G^*\|_F^2+ \biggr(\frac{\sqrt{2\pi}}{(1-2q)\sigma}\|C\|_2\|A^{k-1}\|_2\biggr)^2\\&\hspace{47mm}\times  \max_{i\in \{1,\dots, t/k\}} \mathbb{E}[\|x_{t-ki}\|_2^2].
\end{align*}
We leverage \eqref{expxt} in Lemma \ref{sumxt} to bound $\max_i\mathbb{E}[\|x_{t-ki}\|_2^2]$ and complete the proof.
\end{proof}


\begin{theorem}\label{2conv}
{Suppose that Assumptions \ref{stability}, \ref{maxnorm}, and  \ref{null} hold with $p<\frac{1}{2(k-1)}$, and  let $q:= 1-(1-p)^{k-1}$.
 Then, for all $t\geq 0$ that are multiples of $k$, the estimate $G^{(t)}$ produced by Algorithm \ref{Algorithm 2} at time $t$ satisfies the same upper bound on $\mathbb{E}[\|G^{(t)}-G^*\|_F^2]$ given in Theorem \ref{linconv}. 
}  
\end{theorem}

\begin{proof}
We will use the same decomposition as in \eqref{decomposition}.     
    Using the Polyak step size only considering the positive part, \eqref{noexp} is replaced with
    \begin{align*}
       & \|G^{(t+k)}-G^*\|_F^2 \\&= \|G^{(t)}-G^*\|_F^2 -2\gamma_t \langle g^{(t)}, G^{(t)}-G^*\rangle + \gamma_t^2 \|g^{(t)}\|_F^2  \\&\leq \|G^{(t)}-G^*\|_F^2 -2\gamma_t (f_t(G^{(t)})-f_t(G^*)) + \gamma_t^2 \|g^{(t)}\|_F^2 \\&= \|G^{(t)}-G^*\|_F^2  - (2\alpha-\alpha^2)\frac{(f_t(G^{(t)})-f_t(G^*))^2}{\|g^{(t)}\|_F^2}\\&\hspace{36mm}\times \mathbb{I}\{f_t(G^{(t)})-f_t(G^*)> 0\}.\numberthis\label{lbft2}
    \end{align*}
This ensures that the estimation error still does not increase, and thus we again achieve 
$    [\text{Term~}1] \leq \mathbb{I}(J_t^c)\cdot \|G^{(t)}-G^*\|_F^2,$
as in \eqref{term1}. For $[\text{Term~}2]$, we take  expectation of both sides of \eqref{lbft2} conditioned on $\mathcal{G}_t$ and use the Cauchy-Schwarz inequality to obtain 
    \begin{align*}
           &\mathbb{E}\bigr[\mathbb{I}\{J_t\}\cdot\|G^{(t+k)}-G^*\|_F^2~|~ \mathcal{G}_t\bigr] \\&= \mathbb{I}\{J_t\}\cdot\|G^{(t)}-G^*\|_F^2  \\&\hspace{7mm}- (2\alpha-\alpha^2) \mathbb{E}\left[ \frac{\mathbb{I}\{J_t\}^2\cdot[(f_t(G^{(t)})-f_t(G^*))_+]^2 }{\|g^{(t)}\|_F^2}~\Biggr|~\mathcal{G}_t\right] \\&\leq \mathbb{I}\{J_t\}\cdot\|G^{(t)}-G^*\|_F^2 \\&\hspace{7mm} -  (2\alpha-\alpha^2)\frac{\mathbb{E}\bigr[\mathbb{I}\{J_t\}\cdot (f_t(G^{(t)})-f_t(G^*))_+ ~|~\mathcal{G}_t\bigr]^2}{\mathbb{E}[\|g^{(t)}\|_F^2~|~\mathcal{G}_t]} \\&\leq \mathbb{I}\{J_t\}\cdot\|G^{(t)}-G^*\|_F^2 \\&\hspace{7mm} -  (2\alpha-\alpha^2)\frac{\mathbb{E}\bigr[\mathbb{I}\{J_t\}\cdot  (f_t(G^{(t)})-f_t(G^* )) ~|~\mathcal{G}_t\bigr]^2}{\mathbb{E}[\|g^{(t)}\|_F^2~|~\mathcal{G}_t]},
    \end{align*}
    where $(\cdot)_+ = \max\{\cdot, 0\}$. 
   Note that we arrived at the same expression as \eqref{inequality1}. Thus, we again achieve \eqref{term2} to bound $[\text{Term~}2]$. The remaining arguments follow the same reasoning as in the proof of Theorem \ref{linconv}.
\end{proof}

\begin{remark}
    In Theorems \ref{linconv} and \ref{2conv}, noting that $\max\{X,Y\}$ and $X+Y$ are of the same order, we show that both Algorithms \ref{Algorithm 1} and \ref{Algorithm 2} achieve an expected estimation error of $\mathcal{O}(\max\{ \exp(-t/k^2),\rho^{k-1}\})$, where the first term stems from  $\bigr(1-\frac{(1-2q)^2 (2\alpha-\alpha^2)}{2\pi\cdot mk}\bigr)^{t/(2k)} =\Theta( \exp(-\frac{t}{2k}\cdot \frac{(1-2q)^2 (2\alpha-\alpha^2)}{2\pi \cdot mk} ))$,  using  $-\log(1-x) = \Theta(x)$ for $0<x<\frac{1}{2\pi}$.  Note that this error bound implies that, after a time of $\Omega(k^3\log(1/\rho))$,  the error reaches $\mathcal{O}(\rho^{k-1})$, which corresponds to the estimation error achieved by the $\ell_2$-norm estimator in Theorem~\ref{theorem1}. On the other hand, we were only able to obtain estimates at times that are multiples of $k$, due to the dependence between $\mathbf{U}_{t}^{(k)}$ and $\mathbf{U}_{t+1}^{(k)}$. This streaming estimation is also related to the batch estimation in the sense that the Lipschitz constant of $f_Z(\mathbf{U})$ is $\sqrt{Tk}$ (see \eqref{simple} and \eqref{Tk+1}), which ultimately incurs the recovery time of the $\ell_2$-norm estimator to be $k$ times larger than in the hypothetical case where $\mathbf{U}_{t}^{(k)}$ were independent for all $t$. 
\end{remark}

Although Algorithms \ref{Algorithm 1} and \ref{Algorithm 2} achieve linear convergence until reaching an error of $\mathcal{O}(\rho^{k-1})$, their step sizes $\gamma_t$ are generally not fully accessible since $G^*$ is unknown. 
In the case of deterministic functions, using a geometrically decaying step size without knowledge of $G^*$ can still yield linear convergence \cite{kim2025revisit}; however, this approach does not apply to the stochastic functions considered in this paper.

We now forgo linear convergence and instead aim for sublinear convergence, requiring minimal prior knowledge of $G^*$. Applying classical subgradient descent method (see \cite{boyd2004convex, nesterov2004intro}) with a step size of $\Theta(1/t)$ at time $t$ allows us to bound the minimum suboptimality gap, the difference between the objective values of the current estimate and the true solution. However, this approach bounds only the best-so-far gap, 
leaving the current estimate unreliable and providing no direct bound on the parameter estimation error itself.
 Estimating the current estimation error directly is a central focus of this paper, as our ultimate goal is to use the estimate for controller synthesis with small parameter errors.
To this end, we present Algorithm \ref{Algorithm 3}, which incorporates a projection step into the subgradient descent algorithm, where $\Pi_{\bar G}(\cdot)$ is the projection of a point onto a convex set $\bar G$. 
In the next theorem, we will show that the step size of $\Theta(1/t)$ achieves sublinear convergence until the algorithm has the same exponential order with $\mathcal{O}(\rho^{\frac{k-1}{2}})$, while only requiring an upper bound on the norm of the true Markov parameter matrix.

\begin{algorithm}[!t]
  \caption{Stochastic Projected Subgradient Descent Algorithm with non-increasing step size}

 \begin{flushleft}
        \textbf{Input:} Time Horizon $T$. Time Horizon $T$. Initialization of the Markov parameter matrix $G^{(0)} = 0$. A non-increasing step size rule $(\gamma_t)_{t\geq 0}$.  A convex set $\bar G$ that contains $G^*$. 
{Independent control inputs $u_t \sim \mathcal{N}(0, \sigma^2 I_m)$ for all $t \ge 0$.}      
  \end{flushleft}

  \begin{algorithmic}[1]
   \FOR{Time $t=0,~k,~2k,\dots$}
        \STATE Consider $f_t(G) = 
        \|y_{t+k-1} - G\mathbf{U}_{t+k-1}^{(k)}\|_2$.
        \STATE Select $g^{(t)} \in \partial f_t(G^{(t)})$.

        \STATE Select the step size $\gamma_t\geq 0.$ 
        \STATE Let $G^{(t+k)} = \Pi_{\bar G}(G^{(t)} - \gamma_t g^{(t)}).$
    \ENDFOR 
  \end{algorithmic}
  \begin{flushleft}
        \textbf{Output:} Estimates at times that are multiples of $k$: $G^{(k)}, G^{(2k)}, G^{(3k)}, \dots$.  
  \end{flushleft}

  \label{Algorithm 3}
\end{algorithm}

\begin{theorem}\label{theoremalgo3}
   {Suppose that Assumptions \ref{stability}, \ref{maxnorm}, and  \ref{null} hold with $p<\frac{1}{2(k-1)}$, and  let $q:= 1-(1-p)^{k-1}$.
} Assume that $\|G^*\|_F\leq R$ for a given $R>0$.  
For all $t\geq 0$ that are multiples of $k$, let $G^{(t)}$ denote the estimate produced by Algorithm \ref{Algorithm 3} at time $t$, with the set  $\bar G = \{G: \|G\|_F\leq R\}$ and 
    the step size $\gamma_t = \frac{k\beta}{t+k}$, where $\beta> \frac{\sqrt{2\pi}R}{(1-2q)\sigma}$.
    Then, we have
\begin{align}\label{inductionstep}
\mathbb{E}[\|G^{(t)}-G^*\|_F^2]= \mathcal{O}\Bigr(\frac{\rho^{k-1}\nu}{1-\rho}\cdot\Bigr( \frac{R\sqrt{mk}}{1-2q}+\frac{\rho^{k-1}\nu}{1-\rho}\Bigr)+ \frac{kE}{t+k} \Bigr), 
\end{align}
where $E= \max\left\{\|G^{(0)}-G^*\|_F^2, \frac{\beta^2 \sigma^2 mk}{\beta(1-2q)\sigma /(\sqrt{2\pi}R) -1}\right\}$ and $\nu=\frac{\|C\|_2}{1-2q}\left(\frac{\eta{\sqrt{n}}}{\sigma}+\sqrt{m}\|B\|_2\right)$. 
\end{theorem}

\begin{proof}
Following the step size rule, we have
\begin{align*}
    &\|G^{(t+k)}-G^*\|_F^2 = \|\Pi_{\bar G} (G^{(t)}-\gamma_t g^{(t)})-\Pi_{\bar G}(G^*)\|_F^2\\&\leq \|G^{(t)}-\gamma_t g^{(t)}-G^*\|_F^2\numberthis\label{intere}\\&\leq (\|G^{(t)}-G^*\|_F + \gamma_t \|g^{(t)}\|_F)^2  \\&= \|G^{(t)}-G^*\|_F^2 + 2\gamma_t \|g^{(t)}\|_F \|G^{(t)}-G^*\|_F + \gamma_t^2 \|g^{(t)}\|_F^2,\numberthis\label{generic}
\end{align*}
where we use that $\Pi_{\bar G}(G^*) = G^*$ and the contraction of the projection over the convex set $\bar G$. 
In the remaining proof, we will use the same definition of $J_t$ as the $\mathcal{G}_t$-measurable event that \eqref{xtcondition} happens, and the same decomposition \eqref{decomposition}.

First, $[\text{Term~}1]$ can be bounded by the generic inequality \eqref{generic} as
\begin{align*}
    [\text{Term~}&1] =  \mathbb{E}\bigr[\mathbb{I}\{J_t^c\}\cdot\|G^{(t+k)} -  G^*\|_F^2~|~\mathcal{G}_t\bigr]\\&\leq \mathbb{E}\bigr[\mathbb{I}\{J_t^c\}\cdot\bigr( \|G^{(t)}-G^*\|_F^2 \\&\hspace{10mm}+ 2\gamma_t \|g^{(t)}\|_F \|G^{(t)}-G^*\|_F + \gamma_t^2 \|g^{(t)}\|_F^2\bigr)~|~\mathcal{G}_t\bigr]\\&=\mathbb{I}\{J_t^c\}\cdot (\|G^{(t)}-G^*\|_F^2 \\&\hspace{10mm}+2\gamma_t \sigma\sqrt{mk} \|G^{(t)}-G^*\|_F + \gamma_t^2 \sigma^2 mk),\numberthis\label{term11}
\end{align*}
where we used \eqref{ubgtsq} together with Jensen's inequality to derive an upper bound of $\mathbb{E}[\|g^{(t)}\|_F~|~\mathcal{G}_t]$. 

For [$\text{Term~}2$], we first note that due to the projection and an upper bound of the norm of $G^*$, we have
\begin{align*}
\|G^{(t)}-G^*\|_F \leq \|G^{(t)}\|_F + \|G^*\|_F\leq 2R.
\end{align*}
Then, from \eqref{intere}, we have
    \begin{align*}
        &[\text{Term~}2]=\mathbb{E}\bigr[\mathbb{I}(J_t)\cdot \|G^{(t+k)}-G^*\|_F^2 ~|~\mathcal{G}_t\bigr] 
\\&\leq \mathbb{E}\bigr[\mathbb{I}(J_t)\cdot \|G^{(t)}-\gamma_t g^{(t)}-G^*\|_F^2 ~|~\mathcal{G}_t] 
\\&\leq \mathbb{I}(J_t)\cdot \Bigr(\|G^{(t)}-G^*\|_F^2+\gamma_t^2  \mathbb{E}[\|g^{(t)}\|_F^2~|~\mathcal{G}_t] \Bigr) \\&\hspace{27mm}-2\gamma_t \mathbb{E}[\mathbb{I}(J_t)\cdot \langle g^{(t)}, G^{(t)}-G^*\rangle ~|~\mathcal{G}_t]
\\&\leq \mathbb{I}(J_t)\cdot \bigr(\|G^{(t)}-G^*\|_F^2+\gamma_t^2 \sigma^2mk \bigr) \\&\hspace{27mm}-2\gamma_t \mathbb{E}[\mathbb{I}(J_t)\cdot (f_t(G^{(t)})-f_t(G^*)) ~|~\mathcal{G}_t]
  \\&\leq \mathbb{I}(J_t)\Bigr(\|G^{(t)}-G^*\|_F^2  + \gamma_t^2 \sigma^2 mk\\&\hspace{35mm}-2\gamma_t \|G^{(t)}-G^*\|_F\frac{(1-2q)\sigma}{\sqrt{2\pi}} \Bigr)\\& \leq \mathbb{I}(J_t)\Bigr(\|G^{(t)}-G^*\|_F^2+ \gamma_t^2 \sigma^2 mk \\&\hspace{35mm}-2\gamma_t \frac{\|G^{(t)}-G^*\|_F^2}{2R}\frac{(1-2q)\sigma}{\sqrt{2\pi}} \Bigr) \\&= \mathbb{I}(J_t)\Bigr[\Bigr(1-\gamma_t \frac{(1-2q)\sigma}{\sqrt{2\pi}R}\Bigr) \|G^{(t)}-G^*\|_F^2 + \gamma_t^2 \sigma^2 mk\Bigr],\numberthis\label{term22}
    \end{align*}
    where the third inequality is due to the subgradient inequality $f_t(G^*) - f_t(G^{(t)})\geq\langle g^{(t)}, G^*-G^{(t)}\rangle$, the fourth inequality directly follows from \eqref{lbft} used in the proof of Theorem \ref{linconv}, and the last inequality is due to $\|G^{(t)}-G^*\|\leq 2R$. Integrating \eqref{term11} and \eqref{term22}, we arrive at
    \begin{align*}
       & \mathbb{E}[\|G^{(t+k)}-G^*\|_F^2~|~\mathcal{G}_t] = [\text{Term~}1]+[\text{Term~}2]\\&\hspace{0mm}=\Bigr(1-\gamma_t \frac{(1-2q)\sigma}{\sqrt{2\pi}R}\Bigr) \|G^{(t)}-G^*\|_F^2 + \gamma_t^2 \sigma^2 mk \\&\hspace{10mm}+ \mathbb{I}(J_t^c)\cdot \Bigr( \gamma_t \frac{(1-2q)\sigma}{\sqrt{2\pi}R} \|G^{(t)}-G^*\|_F^2\\&\hspace{35mm}+2\gamma_t \sigma \sqrt{mk} \|G^{(t)}-G^*\|_F \Bigr)\\&\hspace{0mm}\leq\Bigr(1-\gamma_t \frac{(1-2q)\sigma}{\sqrt{2\pi}R}\Bigr) \|G^{(t)}-G^*\|_F^2 + \gamma_t^2 \sigma^2 mk \\&\hspace{5mm}+ \mathbb{I}(J_t^c)\cdot \Bigr( \gamma_t \frac{(1-2q)\sigma}{\sqrt{2\pi}R} \Bigr(\frac{\sqrt{2\pi}}{(1-2q)\sigma}\Bigr)^2\mathbb{E}[\|CA^{k-1} x_{t}\|_2^2~|~\mathcal{G}_t]\\&\hspace{20mm}+2\gamma_t \sigma \sqrt{mk} \frac{\sqrt{2\pi}}{(1-2q)\sigma}\mathbb{E}[\|CA^{k-1} x_{t}\|_2~|~\mathcal{G}_t] \Bigr),
    \end{align*}
    where the equality follows from  $\mathbb{I}(J_t)+\mathbb{I}(J_t^c)=1$, and the inequality is derived from considering the event $J_t^c$ and Jensen's inequality. 
    Since $\mathbb{I}(J_t^c)\leq 1$, taking full expectation and substituting the step size $\gamma_t = \frac{k\beta}{t+k}$ yields 
    \begin{align*}
        &\mathbb{E}[\|G^{(t+k)}-G^*\|_F^2 ]\\&\leq \Bigr(1- \frac{k\beta(1-2q)\sigma}{(t+k) \sqrt{2\pi}R}\Bigr) \mathbb{E}[\|G^{(t)}-G^*\|_F^2] + \frac{\beta^2 \sigma^2 mk^3}{(t+k)^2}\\&\hspace{2mm}+\frac{k\beta}{t+k}\cdot \mathcal{O}\Bigr( \frac{(1-2q)\sigma}{\sqrt{2\pi}R} \Bigr(\frac{\rho^{k-1}\nu}{1-\rho}\Bigr)^2 + 2 \sigma \sqrt{mk}\Bigr(\frac{\rho^{k-1}\nu}{1-\rho}\Bigr)\Bigr),\numberthis\label{induction2}
    \end{align*}
    where we leverage \eqref{expxt} in Lemma \ref{sumxt} to bound  $\mathbb{E}[\|x_{t}\|_2]$ and  $\mathbb{E}[\|x_{t}\|_2^2]$. Now, it follows by induction to prove \eqref{inductionstep}. For a base case, $t=0$ is trivial. For an induction hypothesis, suppose that \eqref{inductionstep} holds for $t$. Then, \eqref{induction2} turns out to be
      \begin{align*}
        &\mathbb{E}[\|G^{(t+k)}-G^*\|_F^2 ]\\&\leq \underbrace{\Bigr(1- \frac{k\beta(1-2q)\sigma}{(t+k) \sqrt{2\pi}R}\Bigr) \cdot\frac{kE}{t+k} + \frac{\beta^2 \sigma^2 mk^3}{(t+k)^2} }_{(a)}\\&\hspace{3mm}+\underbrace{\Bigr(1- \frac{k\beta(1-2q)\sigma}{(t+k) \sqrt{2\pi}R}\Bigr)\cdot \mathcal{O}\Bigr(\frac{\rho^{k-1}\nu}{1-\rho}\cdot \Bigr(\frac{R\sqrt{mk}}{1-2q}+\frac{\rho^{k-1}\nu}{1-\rho}\Bigr)\Bigr)}_{(b)}\\&\hspace{3mm}+\underbrace{\frac{k\beta}{t+k} \cdot \mathcal{O}\Bigr(\frac{(1-2q)\sigma}{\sqrt{2\pi}R} \Bigr(\frac{\rho^{k-1}\nu}{1-\rho}\Bigr)^2 + 2 \sigma \sqrt{mk}\Bigr(\frac{\rho^{k-1}\nu}{1-\rho}\Bigr)\Bigr)}_{(c)},
        \end{align*}
where 
\begin{align}\label{bc}
(b)+(c) = \mathcal{O}\Bigr(\frac{\rho^{k-1}\nu}{1-\rho}\cdot \Bigr(\frac{R\sqrt{mk}}{1-2q}+\frac{\rho^{k-1}\nu}{1-\rho}\Bigr)\Bigr)
\end{align}
and    
\begingroup
  \small     \begin{align*}
         (a)&= \frac{kE}{t+2k} + \frac{k^2}{(t+k)^2}\biggr( \frac{t+k}{t+2k} E - \frac{\beta (1-2q)\sigma}{\sqrt{2\pi}R}E+ \beta^2 \sigma^2 mk\biggr) \\&\leq \frac{kE}{t+2k} + \frac{k^2}{(t+k)^2}\left( \Bigr(1 - \frac{\beta (1-2q)\sigma}{\sqrt{2\pi}R}\Bigr)E+ \beta^2 \sigma^2 mk\right)\\&\leq \frac{kE}{t+2k} + \frac{k^2}{(t+k)^2}\cdot 0 = \frac{kE}{t+2k},\numberthis\label{1/k}
    \end{align*}
    \endgroup%
    where the last inequality is due to the construction of $\beta$ and $E$. Considering \eqref{bc} and \eqref{1/k}  completes the induction step to prove \eqref{inductionstep}.
\end{proof}

\begin{remark}
    In Theorem \ref{theoremalgo3},
    considering the $kE$ term in \eqref{inductionstep},
    we showed that Algorithm \ref{Algorithm 3} achieve an estimation error of $\mathcal{O}(\max\{k/\sqrt{t}, ~k^{1/4} \rho^{\frac{k-1}{2}}\})$ in expectation, where the first term polynomially decays with $t$ and the second term exponentially decays with $k$. Note that this error bound implies that, after a time of $\Omega(k^{3/2}(1/\rho)^k)$, the error reaches $\mathcal{O}(k^{1/4}\rho^{\frac{k-1}{2}})$, which is indeed larger than $\mathcal{O}(\rho^{k-1})$ achieved with Algorithms \ref{Algorithm 1} and \ref{Algorithm 2}. The key distinction in the analysis is that, when either $\langle g^{(t)}, G^*\rangle$ or $f_t(G^*)$ is known,  every subgradient step guarantees a descent in the estimation error deterministically (see \eqref{noexp} or \eqref{lbft2}). In the absence of such information (which is more realistic), however, the estimation error may increase in the worst case (see \eqref{generic}), which results in the relatively weakened exponentially decaying term in $k$. 
\end{remark}
\begin{remark}
    The upper bound $R\ge \|G^*\|_F$ in the theorem 
    {can be estimated a priori using loose physical bounds or system stability parameters,} since any upper bound is acceptable. In particular, we can derive an upper bound for $G^*$ via \eqref{gstar} to arrive at $\|G^*\|_F \leq \sqrt{ \|D\|_F^2 + (k-1)\psi^2\|C\|_F^2 \|B\|_F^2}$, yielding a candidate value for $R$ given coarse bounds on $\|B\|_F, \|C\|_F,$ and $\|D\|_F$.
{Underestimating $R$ such that $R < \|G^*\|_F$ excludes $G^*$ from the projection set and violates the lower bound condition $\beta > \frac{\sqrt{2\pi}R}{(1-2q)\sigma}$, thereby destroying the contraction property. In contrast, choosing a conservative overestimation of $R$ preserves all theoretical guarantees and contraction properties, incurring only a mild increase in the constant factors of the error bound. Thus, choosing a sufficiently large $R$ and $\beta$ is both safe and practically effective.}
\end{remark}

\begin{remark}\label{manybatchremark}
   In Algorithm \ref{Algorithm 3}, we use a stochastic variant that considers the $\ell_2$-norm term regarding a single observation sample at each time. 
   Consider a mini-batch variant that computes the subgradient using the averaged $\ell_2$-norm terms over a fixed number $(>1)$ of observation samples (\text{e.g.}, $f_t(G) = \frac{1}{\mathcal{B}}\sum_{i\in B_t} \|y_i  - G\mathbf{U}_i^{(k)}\|_2$, where $B_t\subseteq \{0, \dots, t\}$ is a mini-batch of size $\mathcal{B}$ at time $t$), which may provide additional benefits at a constant-factor computational cost. In particular, it could improve \eqref{ubgtsq} by reducing the expected squared norm of the subgradient when the observation samples are sufficiently separated over time. 
    However, recycling past data within a single trajectory cannot fully eliminate dependence, since current estimates still rely on prior samples. A common remedy is to first collect a sufficiently large dataset along the trajectory to ensure high concentration around expectation (see \cite{oymak2019stochastic} for a similar single-trajectory analysis), and then to randomly select multiple trajectory samples from it at each subgradient descent step. This approach yields estimates at all times, rather than only at multiples of $k$, since the accumulated data effectively serves as an oracle of control inputs that provide sufficient excitation (as in \eqref{lbft}) while keeping the squared norm bounded (as in \eqref{ubgtsq}). 
   Accordingly, 
   a mini-batch algorithm is viable once enough data have been collected.
\end{remark}

\section{Hybrid Estimation: System Identification Framework}\label{sec:framework}

In this section, we provide a general framework to recover the true system via the batch and the streaming estimation approaches. 
In particular, we use the estimated Markov parameter matrix obtained via the approaches in Sections \ref{sec:l2est} and \ref{sec:subg}, respectively, and concurrently recover $A,B,C,D$ from $G^* = [D~ CB ~ CAB ~ \dots ~ CA^{k-2}B]$. This is a nonconvex problem, resulting in infinitely many solutions up to a similarity transformation. To address this issue, it turns out that the balanced truncation can be recovered up to order $\lfloor\frac{k}{2}\rfloor$ from $G^*$ via the Ho-Kalman algorithm \cite{hokalman1966kalman}. The analysis consists of constructing Hankel matrix and the resulting $d$-order balanced truncated model defined below. 

\medskip 

\begin{definition}[Hankel Matrix]\label{hankeldef}
    The $(\alpha, \beta)$-dimensional Hankel matrix for $(A,B,C)$ is defined as 
    \[
    \mathcal{H}_{\alpha,\beta}=  \begin{bmatrix}
        CA^\alpha B & CA^{\alpha+1}B & \cdots & CA^{\alpha+\beta-1}B \\ CA^{\alpha+1}B & CA^{\alpha+2}B & \cdots & CA^{\alpha+\beta}B &  \\ \vdots & \vdots & \ddots &\vdots \\ CA^{\alpha+\beta-1}B & CA^{\alpha+\beta}B & \cdots& CA^{\alpha+2\beta-2}B
    \end{bmatrix}. 
    \]
\end{definition}


\begin{definition}[$d$-order balanced truncated model]\label{dorder}
    Let the singular value decomposition (SVD) of the matrix $\mathcal{H}_{0,\infty}$ be given as $U\Sigma V^T$, where $\Sigma\in \mathbb{R}^{n\times n}$ is a diagonal matrix with singular values $\sigma_1\geq \dots\geq \sigma_n\geq 0$. Then for any $d\in \{1,\dots n\}$, the $d$-order balanced truncated model is defined as
    \begin{subequations}
         \begin{align*}
        &C^{(d)} = (U\Sigma^{1/2})_{[1:r], [1:d]}, ~ B^{(d)} = (\Sigma^{1/2}V^T)_{[1:d], [1:m]} \\ & A^{(d)} = (U\Sigma^{1/2})_{[1:\infty], [1:d]}^{\dagger} (U\Sigma^{1/2})_{[r+1:\infty], [1:d]},
    \end{align*}
    \end{subequations}
    where $(\cdot )^\dagger$ denotes the pseudoinverse and $X_{[n_1:n_2], [m_1:m_2]}$ denotes the submatrix of $X$ consisting of the $n_1^\text{th}$ to the $n_2^\text{th}$ rows and the $m_1^\text{th}$ to the $m_2^\text{th}$ columns.
\end{definition}

We use an estimated Markov parameter matrix obtained in the previous sections to produce 
the estimates $\hat A^{(d,t)}, \hat B^{(d,t)}, \hat C^{(d,t)}, \hat D^{(t)}$ to   
recover $A^{(d)}, B^{(d)}, C^{(d)}, D$ at each time $t$. To be specific, the first $r\times m$ submatrix of the estimated Markov parameter matrix will exactly be $\hat D^{(t)}$; the next submatrices will be the estimations of $CB, CAB, \dots, CA^{k-2}B$. Then, one can obtain the estimated $d$-order balanced truncated model via Definition \ref{dorder} by replacing each component of the Hankel matrix with estimations of $CB, CAB, \dots, CA^{k-2}B$, using the zero matrix for the components $CA^{k-1} B , CA^{k}B, \dots$, and perform SVD. 

The works \cite{sarkar2021finite, oymak2022revisit} provided an analysis that the estimations of the balanced truncated model are close enough to the true model if the true Hankel matrix and the estimated Hankel matrix are close enough. Thus, it suffices to show that the accuracy of the estimated Hankel matrix. 
Algorithm \ref{Algorithm 4} integrates batch and streaming estimation approaches, by using either batch or streaming estimates of the Markov parameter matrix, whichever is more suitable at each time.
More precisely, recall that there exists a recovery time $T^*$ such that the $\ell_2$-norm estimator achieves an estimation error of $\mathcal{O}(\rho^k)$; for $T<T^*$, Algorithms \ref{Algorithm 1}, \ref{Algorithm 2}, or \ref{Algorithm 3}  yield errors with an exponentially decaying term in $k$ plus a term in $t$ that decays exponentially or polynomially. 
The following theorem provides the estimation error bound for the Hankel matrix obtained from Algorithm \ref{Algorithm 4}, where Line \ref{line3} implements Algorithm \ref{Algorithm 3}. 
The result readily extends to the case in which Algorithm \ref{Algorithm 1} or \ref{Algorithm 2} is invoked in place of Algorithm \ref{Algorithm 3}.


\begin{algorithm}[!t]
  \caption{System Identification and Adaptive Control Framework using batch and streaming estimates}

 \begin{flushleft}
        \textbf{Input:} Recovery time $T^*$, given by \eqref{mintime},  required for the $\ell_2$-norm estimator.  Initial Markov matrix estimate $G^{(0)} = 0$. The balanced truncated model order $d\leq \lfloor \frac{k}{2}\rfloor$. 
  \end{flushleft}

  \begin{algorithmic}[1]
    \FOR{Time $t=1,2,\dots,T^*$}
        
        \IF{$t<T^*$}
            \STATE Use Algorithm \ref{Algorithm 1}, \ref{Algorithm 2}, or \ref{Algorithm 3} depending on the available information and obtain $G^{(t)}$.\alglinelabel{line3}
        \ELSE
            \STATE Solve the $\ell_2$-norm estimator \eqref{l2est} to obtain $\hat G^*_{T^*}$, 
            
        \ENDIF   
\STATE Construct the estimated Hankel matrix $\mathcal{\hat H}^{(t)}$ using $G^{(t)}$ (if $t<T^*$) or $\hat G^*_{T^*}$ (if $t=T^*$). Each entry is formed from the available estimates of the Markov parameters ($CB, \dots, CA^{k-2}B$), while unavailable components ($CA^{k-1}B, CA^{k}B, \dots$) are set to zero.

        \STATE        
        Derive the estimated balanced truncated model $\hat A^{(d,t)},\hat B^{(d,t)},\hat C^{(d,t)}, \hat D^{(t)}$ from the SVD of the estimated Hankel matrix $\mathcal{\hat H}^{(t)}$. 

        \STATE Based on $\hat A^{(d,t)}, \hat B^{(d,t)},\hat C^{(d,t)}, \hat D^{(t)}$, synthesize a robust adaptive controller for stability/performance guarantees.\alglinelabel{line9}
        
    \ENDFOR 
  \end{algorithmic}

  \label{Algorithm 4}
\end{algorithm}

\begin{theorem}\label{hankelthm}

{Suppose that Assumptions \ref{stability}, \ref{maxnorm}, and  \ref{null} hold with $p<\frac{1}{2(k-1)}$.
}
Consider the variant of Algorithm \ref{Algorithm 4} that invokes  Algorithm \ref{Algorithm 3} at Line \ref{line3}.
When $t<T^*$, the estimation errors of $D$ and the Hankel matrix, respectively, are given as
\begin{align*}
& \bullet ~\mathbb{E}[\|D-\hat D^{(t)}\|_2] =  \mathcal{O}_{\rho, k, t}\biggr(\frac{k^{1/4}\rho^{\frac{k-1}{2}}}{\sqrt{1-\rho}}+\frac{k}{\sqrt{t}} \biggr), 
 \\& \bullet~ \mathbb{E}[\|\mathcal{H}_{0,\infty} -  \mathcal{\hat H}^{(t)}\|_2] =\mathcal{O}_{\rho, k, t}\biggr(\frac{ \rho^{ k/2} }{1-\rho^2}+\frac{k^{3/4}\rho^{\frac{k-1}{2}}}{\sqrt{1-\rho}}+\frac{k^{3/2}}{\sqrt{t}} \biggr),
\end{align*} 
where $\mathcal{O}_{\rho,k,t}(\cdot)$ denotes the big-$\mathcal{O}$ notation hiding all factors independent of $\rho,k,t$. Moreover, when $t=T^*$, we have 
\begin{align*}
& \bullet ~\|D-\hat D^{(t)}\|_2 =  \mathcal{O}_{\rho, k, t}\biggr(\frac{\rho^{k-1}}{1-\rho}\biggr), 
 \\& \bullet~ \|\mathcal{H}_{0,\infty} -  \mathcal{\hat H}^{(t)}\|_2 =\mathcal{O}_{\rho, k, t}\biggr(\frac{ \rho^{ k/2} }{1-\rho^2}+\frac{k^{1/2}\rho^{k-1}}{1-\rho} \biggr)
\end{align*} 
with probability at least $1-\delta$.
\end{theorem}

\begin{proof}
For convenience, let $\tilde G^{(t)}$   be $G^{(t)}$ for all $t<T^*$, and $\hat G^*_{T^*}$ for $t=T^*$. Then, from Theorem \ref{theoremalgo3}, we have
\begin{align}\label{beforeT}
    &\mathbb{E}[\|G^* - \tilde G^{(t)}\|_2]=\mathcal{O}_{\rho, k, t}\biggr(\frac{k^{1/4}\rho^{\frac{k-1}{2}}}{\sqrt{1-\rho}} + \frac{k}{\sqrt{t}}\biggr), \quad \forall t < T^*,
\end{align}
which follows from $\|\cdot \|_2 \leq \|\cdot \|_F$ and Jensen's inequality. From Theorem \ref{theorem1}, we also have 
\begin{align}\label{afterT}
    \|G^*-\tilde G^{(t)}\|_2 =\mathcal{O}_{\rho,k, t}\biggr(\frac{\rho^{k-1}}{1-\rho}\biggr), \quad t=T^*,
\end{align}
with probability at least $1-\delta$. 
Considering these relations, it now suffices to bound $\|D-\hat D^{(t)}\|_2$ and $\|\mathcal{H}_{0,\infty} -  \mathcal{\hat H}^{(t)}\|_2 $ in terms of $\|G^*- \tilde G^{(t)}\|_2$.

   First, note that $D$ and $\hat D^{(t)}$ correspond to the first $r\times m$ submatrices of $G^*$ and  $\tilde G^{(t)}$, respectively. Hence, we have $\|D-\hat D^{(t)}\|_2 \leq \|G^*- \tilde G^{(t)}\|_2$, which immediately provides the estimation error bound for $D$.

    Now, for the estimation error of the Hankel matrix, let $d:=\lfloor \frac{k}{2}\rfloor$ and     
    define $\mathcal{\bar H}_{0,d}$ as the zero-padded matrix of $\mathcal{H}_{0,d}$, where the right and bottom parts are extended infinitely with zeros, with $\mathcal{H}_{0,d}$ as its leading principal submatrix. Then, by the triangle inequality, we have
    \begin{align*}
        \|\mathcal{H}_{0,\infty} -  \mathcal{\hat H}^{(t)}\|_2\leq \underbrace{\|\mathcal{H}_{0,\infty} -  \mathcal{\bar H}_{0,d}\|_2}_{[\text{Term~I}]}+\underbrace{\|\mathcal{\bar H}_{0,d} -  \mathcal{\hat H}^{(t)}\|_2}_{[\text{Term~II}]}.
    \end{align*}
    For $[\text{Term~I}]$, note that 
$\mathcal{H}_{0,\infty} -\mathcal{\bar { H}}_{0,d} =\begin{bmatrix}
    0 & H_{12}\\ H_{21}&H_{22}
\end{bmatrix},$
where $\begin{bmatrix}
    H_{21}&H_{22}
\end{bmatrix}=\begin{bmatrix}
    H_{12}\\H_{22}
\end{bmatrix}=\mathcal{H}_{d,\infty}$. Considering that the squared spectral norm of a matrix is bounded by the sum of the squared spectral norms of its submatrices, we have 
\begin{align*}
    \|&\mathcal{H}_{0,\infty} -\mathcal{\bar { H}}_{0,d} \|_2 \leq(\|\begin{bmatrix}
    H_{21}&H_{22}
\end{bmatrix}\|_2^2 + \|H_{12}\|_2^2)^{1/2}\\&\hspace{3mm}\leq 
\sqrt{2} \|\mathcal{H}_{d,\infty}\|_2 = \sqrt{2}\left\|
\begin{bmatrix}
    C \\CA \\\vdots
\end{bmatrix}A^d
\begin{bmatrix}
    B &AB&\cdots
\end{bmatrix}
\right\|_2\\&\hspace{3mm}\leq \sqrt{2} \bigr(\sum_{i=0}^\infty \|CA^{i}\|_2^2\bigr)^{1/2}\cdot \|A^d\|_2 \cdot \bigr(\sum_{i=0}^\infty \|A^{i}B\|_2^2\bigr)^{1/2}\\&\hspace{3mm}\leq \sqrt{2}\psi^3 \frac{ \rho^d \|C\|_2  \|B\|_2 }{1-\rho^2}\numberthis\label{firstbound}
\end{align*}
For $[\text{Term~II}]$, note that each block matrix consisting of rows $(i-1)r+1$ to $ir$ of $  \mathcal{\bar H}_{0,d} - \mathcal{\hat H}^{(t)}$ for $i=1, \dots, d$ is a submatrix of $G^* - \tilde G^{(t)}$ by the construction of Hankel matrices. 
Thus, we have
\begin{equation}\label{secondbound}
    \|\mathcal{\bar H}_{0,d} - \mathcal{\hat H}^{(t)}\|_2 \leq \sqrt{d} \|G^* - \tilde G^{(t)}\|_2.
\end{equation}
Recall that $d=\lfloor \frac{k}{2}\rfloor$, and integrate \eqref{firstbound} and \eqref{secondbound} to yield
\begin{align*}
    \|\mathcal{H}_{0,\infty} -  \mathcal{\hat H}^{(t)}\|_2 =\mathcal{O}_{\rho, k, t}\biggr(\frac{\rho^{k/2}}{1-\rho^2} + \sqrt{k}\|G^*-\tilde G^{(t)}\|_2\biggr), 
\end{align*}
Applying \eqref{beforeT} to this bound for $t<T^*$ and \eqref{afterT} for $t=T^*$ completes the proof.
\end{proof}

\begin{remark}
 The proof of Theorem \ref{hankelthm} establishes that the estimation errors for both $D$ and the Hankel matrix are affinely proportional to $\|G^*-\tilde G^{(t)}\|_2$; i.e., the estimation error of the Markov parameter matrix at each time $t$.  Note that in Algorithm \ref{Algorithm 4}, we leverage the \textit{streaming estimation} at the initial stage ($t<T^*$) to guarantee control over the expected $\|G^*-\tilde G^{(t)}\|_2$, which decreases gradually over time $t$ via Algorithms \ref{Algorithm 1}, \ref{Algorithm 2}, or \ref{Algorithm 3}, thereby ensuring that the estimation error of the Hankel matrix remains appropriately bounded at earlier times. 
     When the time exceeds $T^*$, which is required for the batch data the $\ell_2$-norm estimator uses, the estimation error of the Hankel matrix arrives at $\mathcal{O}(\rho^{k/2} + k^{1/2}\rho^{k-1})$ with high probability, indicating that the error ultimately decays exponentially as $k$ grows under the \textit{batch estimation} regime.
Unlike the streaming approach leveraged for $t<T^*$, batch estimation typically incurs significant computational overhead when performed repeatedly, whereas Algorithm \ref{Algorithm 4} requires computing the $\ell_2$-norm estimator only once and incurs negligible overhead, since the estimator achieves the best achievable error with high probability for all $t\geq T^*$.
\end{remark}

\begin{remark}
Line \ref{line9} of Algorithm \ref{Algorithm 4} highlights that we can synthesize a controller based on the estimated balanced truncated model. 
We note that the accuracy of this system identification  contributes to designing a well-performed controller. In particular, when the model mismatch is sufficiently small, the adaptive control approach that designs a controller based on the estimated model and applying it to the true system is often effective.
This principle has direct implications for stability. Given an adaptive controller designed with a positive stability margin,  the stability of the true closed-loop system is still guaranteed if the model error is smaller than this margin.
Similarly for performance guarantee, even though an adaptive controller will yield suboptimal performance when the model has a small error, the performance loss (\textit{e.g.}, suboptimality gap) remains bounded within an acceptable limit. Detailed analysis of robust stability and robust performance under small model perturbations can be found in the literature on the robust control theory (see \cite{zhourobust}).
\end{remark}

\section{Numerical Experiments}\label{sec:numexp}
In this section, we provide numerical experiments to demonstrate our theoretical findings in the previous sections. 

\begin{example}[Batch estimation]\label{example1}
    In this example, we illustrate the results in Section \ref{sec:l2est}, and show how the $\ell_2$-norm estimator outperforms the least-squares method and the $\ell_1$-norm estimator. Recall that Theorem \ref{theorem1} shows that the $\ell_2$-norm estimator shows an estimation error of $\mathcal{O}(\rho^{k-1})$ when $T= \Omega(k^2 m+ k \log(\frac{1}{\delta}))$ under $p<\frac{1}{2(k-1)}$; while our preliminary version shows that the $\ell_1$-norm estimator arrives at an estimation error of $\mathcal{O}(\sqrt{r} \cdot \rho^{k-1})$ when $T=\Omega(k^2 m+ k \log(\frac{r}{\delta}))$. 
Note that the least-squares method cannot achieve a small enough error unlike the aforementioned two estimators due to the lack of robustness to adversarial attacks.

\begin{figure}[t]
     \centering
     \begin{subfigure}[b]{0.235\textwidth}
         \centering
    \includegraphics[width=\textwidth]{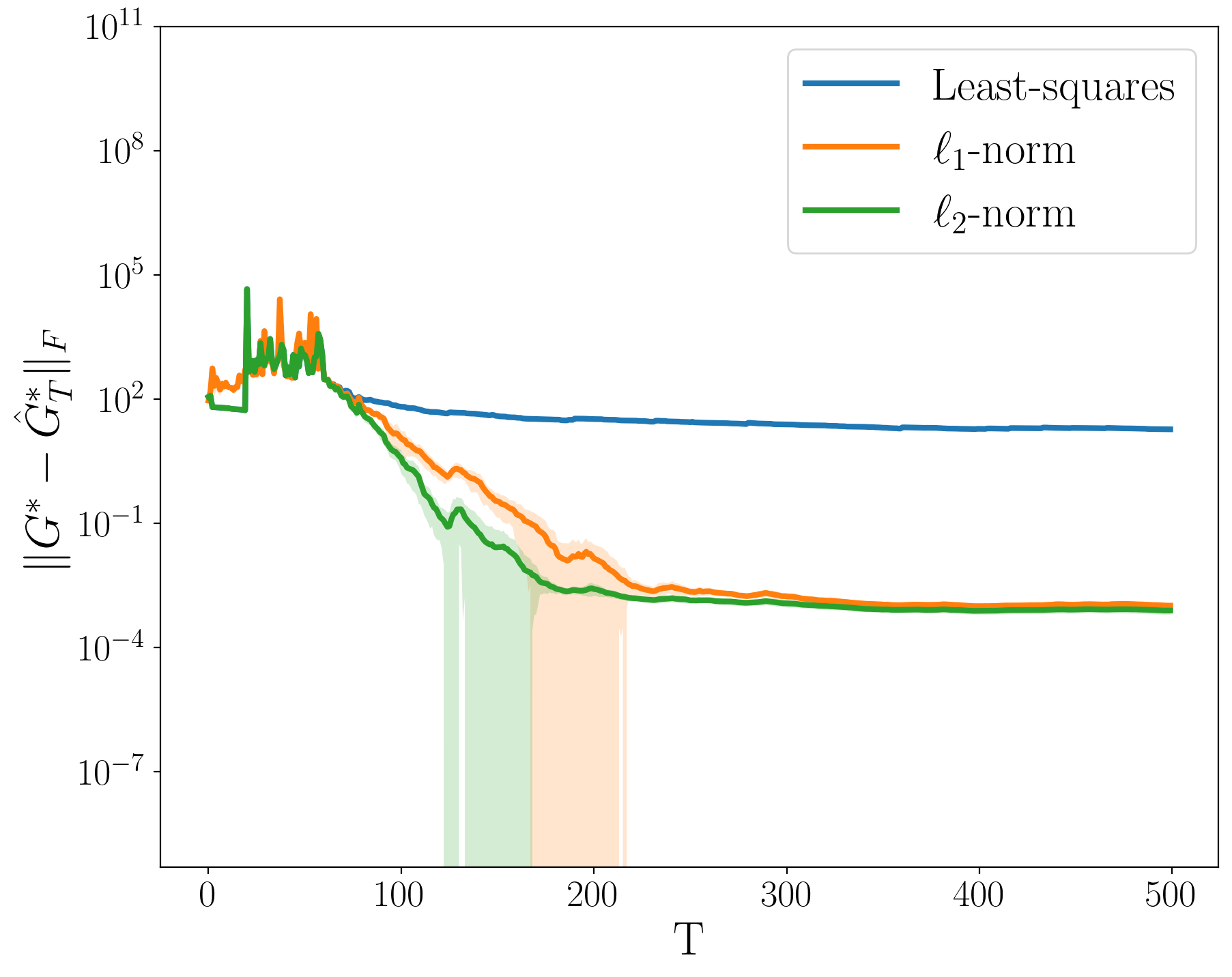}
         \caption{$k=10$}
         \label{k10}
     \end{subfigure}
     \begin{subfigure}[b]{0.235\textwidth}
         \centering
    \includegraphics[width=\textwidth]{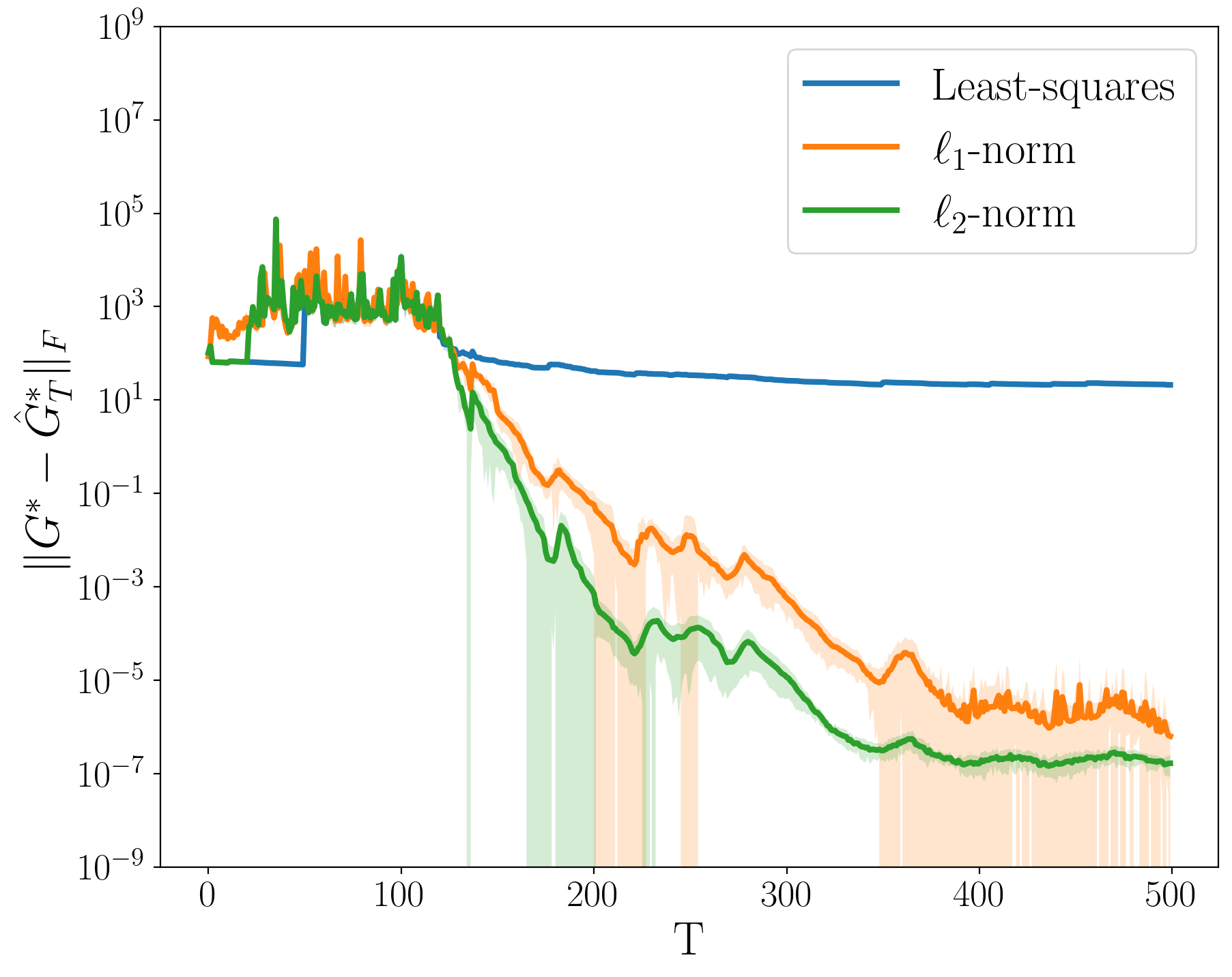}
         \caption{$k=20$}
         \label{k20}
     \end{subfigure}
        \caption{{{\fontsize{9.5pt}{11.5pt}\selectfont Estimation error for the Markov parameter matrix of order $k=10, 20$ (batch estimation): $\ell_2$-norm estimator vs. $\ell_1$-norm estimator vs. least-squares under adversarial attacks. Solid line = mean, Shaded area = $\pm2$ standard deviation. We report the average runtimes of the experiment of $k=20$ up to $T=500$: $\ell_2$-norm—247.59s, $\ell_1$-norm—282.18s, Least-squares—317.36s.}}}
        \label{ex1}
\end{figure}

We use $n=300$, $m=6$, and $k\in\{10,20\}$ for experiments. We generate the matrix $A$ by selecting all entries from Uniform$[-1,1]$ and rescaling it to satisfy $\|A\|_2 = 0.6$. 
{We consider a networked system of $N = 30$ agents, where each agent $i \in \{1,  \dots, 30\}$ observes the state entries $10i -9, \dots, 10i$, providing an $10$-dimensional partial observation ($r=10$) of the global state $x_t \in \mathbb{R}^{300}$.}
The initial state is set to a vector of $1000$s, the control inputs at each time is designed to follow $N(0, 100I_{m})$, and the attack time probability is set to $p=\frac{1}{2k}$. The attack $w_t$ is sub-Gaussian with covariance $25I_{n}$ and a mean vector whose entries are either $300$ or $1000$, depending on the sign of the corresponding coordinate of $x_t$. 

{We report the estimation errors across agents $1, \dots, 10$, where the solid line represents the mean and the shaded area indicates $\pm 2$ standard deviation across agents.}
Fig.~\ref{ex1} shows the estimation error on a log scale over time, where the least-squares method fails to recover the Markov parameter matrix, resulting in an error of around $10^2$. In contrast, the  $\ell_2$-norm and $\ell_1$-norm estimators yield a decreasing error to arrive at a sufficiently small error for both $k=10$ and $20$. However, one can observe that the $\ell_2$-norm estimator outperforms the $\ell_1$-norm estimator in terms of both the recovery time and the ultimate error. Moreover, for both $\ell_2$-norm and $\ell_1$-norm estimators, a larger $k$ results in a smaller error, although a longer time is required for the convergence, which strongly supports Theorem \ref{theorem1}.
\end{example}

\begin{example}[Streaming estimation]
    In this example, we illustrate the results in Section \ref{sec:subg}, and show how the stochastic gradient descent can provide the estimates at all time steps. The experimental settings are identical to those in Example \ref{example1}, except that here we present results only for $k=20$. 
    First, we run a stochastic subgradient descent using Algorithms \ref{Algorithm 1}, \ref{Algorithm 2}, and \ref{Algorithm 3}. Recall that Algorithms \ref{Algorithm 1} and \ref{Algorithm 2} assume knowledge of the true Markov parameter matrix $G^*$, whereas Algorithm \ref{Algorithm 3} does not. In Fig.~\ref{sto}, all three algorithms {are compared against the batch $\ell_2$-norm estimator baseline, which is evaluated using the full dataset accumulated up to time $t$. Noting that the batch estimator quickly  plateaus  around $t=1000$, each streaming algorithm consistently reduces the estimation error as more data arrive.} Algorithms \ref{Algorithm 1} and \ref{Algorithm 2} achieve linear convergence {and ultimately reach the error level of the batch estimator}, while Algorithm \ref{Algorithm 3} still shows a decrease consistent with sublinear convergence $\mathcal{O}(k/\sqrt{t})$ despite the lack of knowledge of $G^*$. 

    \begin{figure}[t]
     \centering
     \begin{subfigure}[b]{0.235\textwidth}
         \centering
    \includegraphics[width=\textwidth]{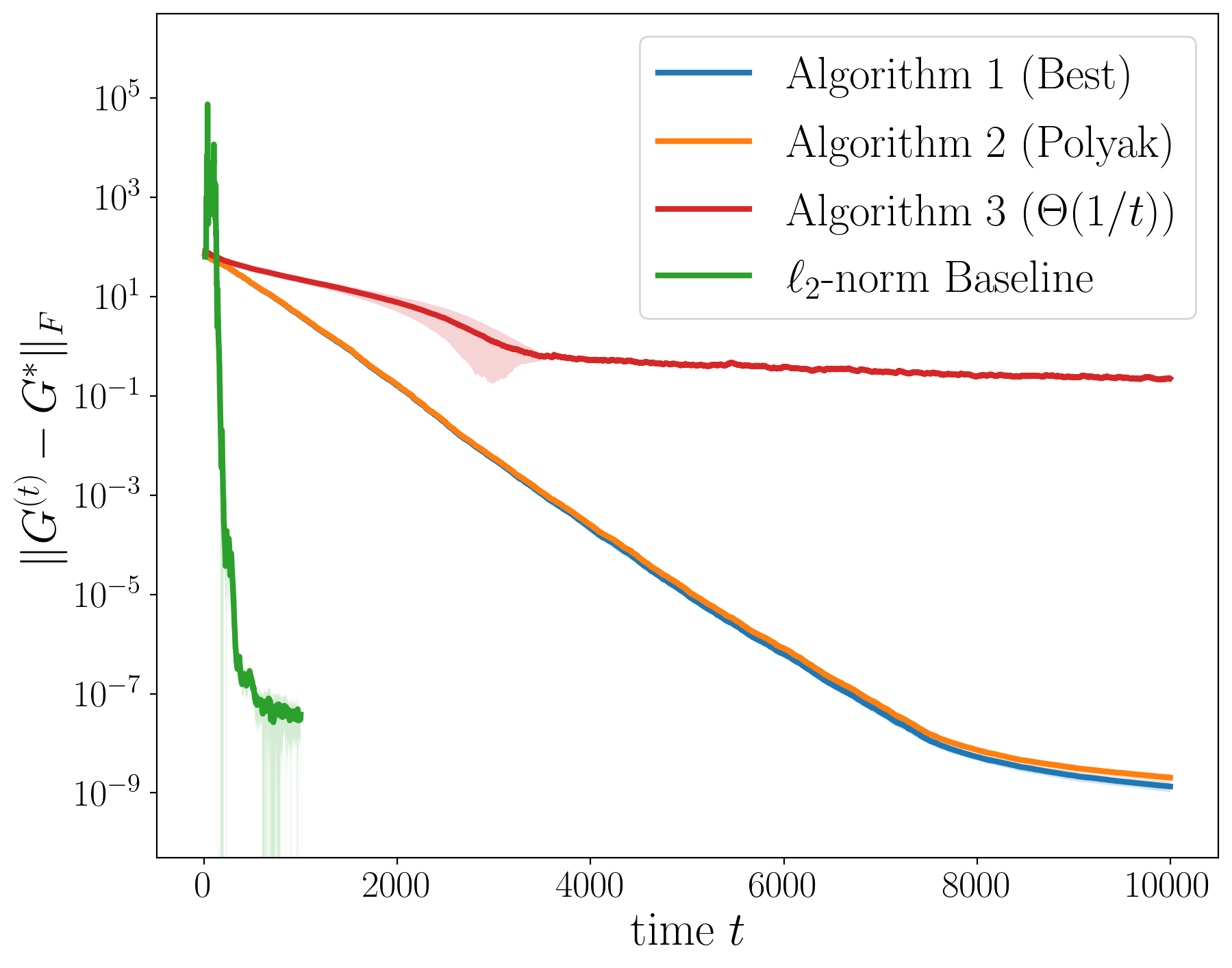}
         \caption{Stochastic  descent algorithm}
         \label{sto}
     \end{subfigure}
     \begin{subfigure}[b]{0.235\textwidth}
         \centering
    \includegraphics[width=\textwidth]{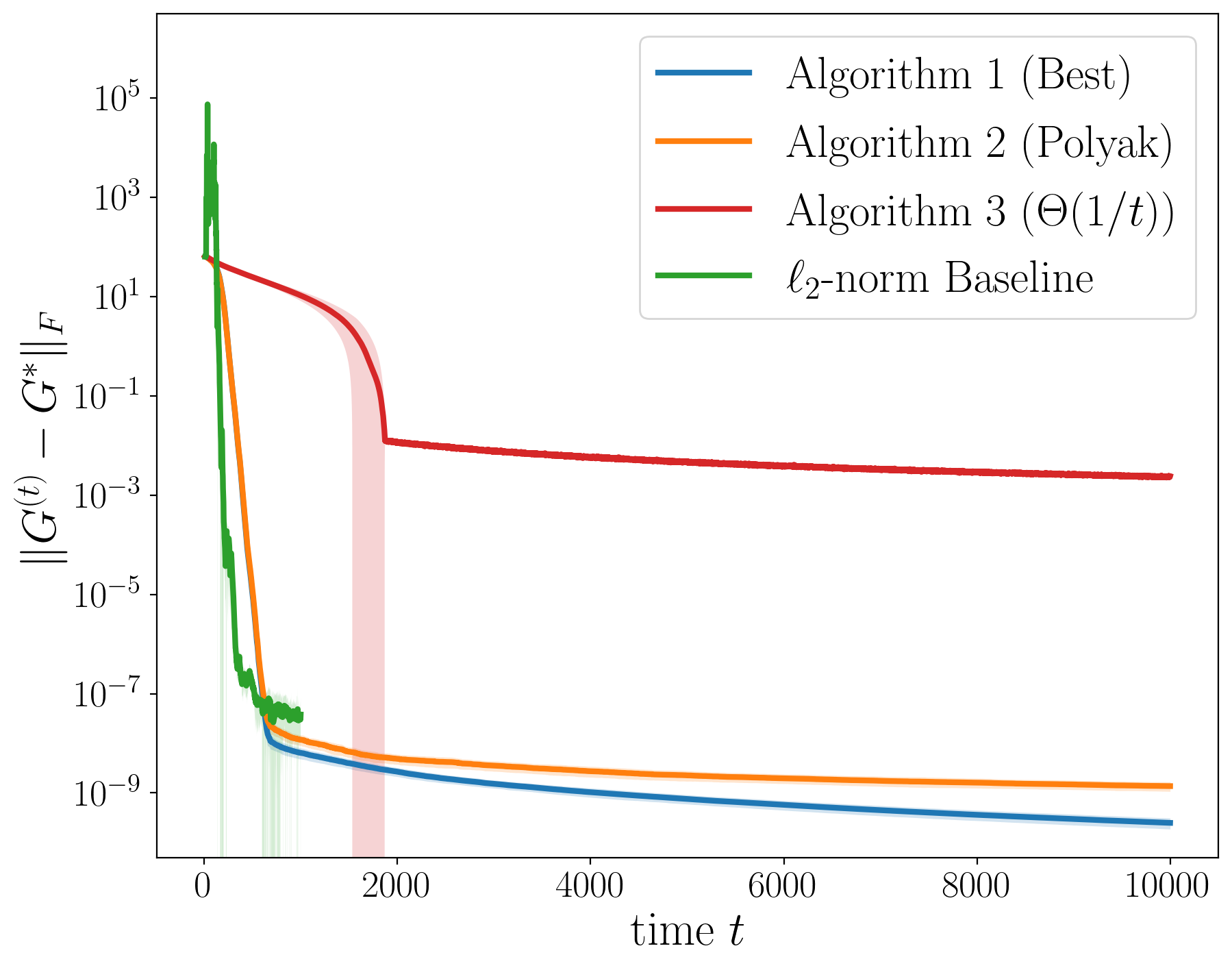}
         \caption{Mini-batch descent algorithm}
         \label{minib}
     \end{subfigure}
        \caption{ {{\fontsize{9.5pt}{11.5pt}\selectfont Estimation error for the Markov parameter matrix of order $k=20$ (streaming estimation): Best step size vs. Polyak step size vs. $\Theta(1/t)$ vs. $\ell_2$-norm estimator baseline (up to time 1000) under adversarial attacks. Solid line = mean, Shaded area = $\pm2$ standard deviation. We report the average runtimes of running all three step size rules up to $t=10000$: Stochastic—5.89s, Mini-batch—9.81s.}}}
        \label{ex2}
\end{figure}

Then, we additionally run a mini-batch subgradient descent as discussed in Remark \ref{manybatchremark}, applying all three algorithms with a batch size of $100$. At each time $t$, we randomly select $\min\{t,100\}$ observation trajectory samples from time steps $0, \dots, t$ to compute the average of the $\ell_2$-norm terms, rather than using a single $\ell_2$-norm term.
The descent rate improves remarkably with the increased batch size. Interestingly, although the runtime of the mini-batch approach is theoretically $100$ times that of the stochastic approach, vectorization reduces this overhead to only $1.7$ times. This suggests that the mini-batch approach is preferable over the stochastic approach despite the slight increase in computational cost.
We also note that the computational runtimes reported above are significantly lower than those of the batch estimators; in particular, the batch estimation experiments were conducted up to time $500$, whereas the streaming estimation experiments were run up to time $10000$ and still required much less time.

{Importantly, the batch $\ell_2$-norm estimator baseline in Fig.~\ref{ex2} suffers a sudden spike in estimation error at an initial stage}, indicating that its estimates are unreliable during early times. In contrast, the streaming estimation in Fig.~\ref{ex2} (particularly \ref{minib}) demonstrates a decreasing behavior from the beginning. Accordingly, the \textit{batch and streaming estimation approaches should be integrated}: the streaming approach should be initially leveraged to ensure both early-time descent and manageable computational complexity, and the batch approach can be applied once the system designer is confident that the recovery time of the $\ell_2$-norm estimator has been reached.
\end{example}

\section{Conclusion}\label{sec:conclusion}

In this paper, we aim to obtain an accurate identification of the Markov parameter matrix of order $k$, denoted by $G^*$, where the probability $p$ of having an attack at each time is {smaller than $0.5$}. First, we study the batch $\ell_2$-norm estimator in terms of control inputs and observations, and show that there exists a finite time after which the estimator uses the accumulated batch data to achieve an estimation error that decays exponentially {as $p$ decreases}. {We further analyze a low-dimensional subspace attack scheme where the exponential decay rate is decoupled from the attack probability $p$.}  Second, we study the stochastic projected subgradient descent method for the streaming data and provide theoretical guarantees on the expected estimation error at all times $t$, using a step size of $\Theta(1/t)$. The resulting error is bounded by the larger of $\mathcal{O}(k/\sqrt{t})$ and {the baseline error of the batch $\ell_2$-norm estimator} when little is known about $G^*$. 
This stochastic algorithm provides computational efficiency and yields estimates whose error bounds decrease over time in an initial stage, even when 
the data has not yet been sufficiently collected, whereas the $\ell_2$-norm estimator requires a minimum amount of data to work effectively but achieves the best error once the required data is available.
These provable estimation errors and operational properties of both batch and streaming approaches carry over to the estimation error of the Hankel matrix and the balanced truncations of the true system. This work addresses the possibility of integrating streaming and batch estimations to synthesize a robust adaptive controller at any point during the system identification process.

\appendix

\subsection{Proof of Lemma \ref{gauscon}}\label{prooflemma1}

\begin{proof}
   When $f(X)$ has a Lipschitz constant of $L$, then the function $\tilde f(X):= f(\mu + \sigma X)$ has a Lipschitz constant of $\sigma L$ in the sense that 
\begin{align*}
   | f(\mu + \sigma X_1 ) - f(\mu + \sigma X_2 )| &\leq L\| (\mu + \sigma X_1) - (\mu + \sigma X_2)\| \\&= \sigma L \|X_1 -X_2\|. 
\end{align*}
Given that $X\sim N(0, I_M)$, we can apply Theorem 5.2.3 in \cite{vershynin2025high} to $\tilde f(X)$ and arrive at
\begin{align*}
    \|\tilde f(X) - \mathbb{E}[\tilde f(X)]\|_{\psi_2} =\mathcal{O}(\sigma L).
\end{align*}
Observing that $\tilde f(X)$ has the same distribution with $f(\mathbf{U})$ completes the proof. 
\end{proof}

\subsection{Proof of Lemma \ref{sumxt}}\label{prooflemma2}

\begin{proof}
Note that $\|\|x_0\|_2 \|_{\psi_2} \leq \eta{\sqrt{n}}, \|\|w_t\|_2\|_{\psi_2} \leq \eta{\sqrt{n}}, $ and $\|\|u_t\|_2\|_{\psi_2} \leq \sigma\sqrt{m}$  follow from Assumption \ref{maxnorm} and the Gaussianity of $u_t$. 
Using the system dynamics \eqref{sysd} yields
\begin{align*}
    \|\|x_t&\|_2\|_{\psi_2} = \Bigr\|\Bigr\|A^t x_0 + \sum_{i=0}^{t-1} (A^{t-1-i}Bu_i+A^{t-1-i} w_{i})\Bigr\|_2 \Bigr\|_{\psi_2} \\& \leq \|A^t\|_2\cdot\|\|x_0\|_2\|_{\psi_2} \\&\hspace{5mm}+ \sum_{i=0}^{t-1} \|A^{t-1-i}\|_2 \bigr(\|B\|_2\cdot \|\|u_i\|_2\|_{\psi_2} + \|\|w_i\|_2\|_{\psi_2}\bigr) \\&= \frac{\psi}{1-\rho} \cdot \mathcal{O}(\eta {\sqrt{n}} + \sigma \sqrt{m} \|B\|_2) 
    \numberthis\label{xtsums}
\end{align*}
due to the triangle inequality and geometric sum with Assumption \ref{stability}. The sub-Gaussian property that $\mathbb{E}[X^2]=\mathcal{O}(\|X\|_{\psi_2}^2)$ for a sub-Gaussian variable $X$ then concludes \eqref{expxt}. We can also sum up $\|x_t\|_2$ over $t=0, \dots, T-1$ to yield
\begin{align*}
        &\sum_{t=0}^{T-1} \|x_t\|_2 
        \le\sum_{i=0}^{\infty} \|A^i\|_2   \Bigr[\|x_0\|_{2} + \sum_{t=0}^{T-2} (\|w_t\|_{2} +  \|Bu_t\|_{2})\Bigr]\\&\hspace{8mm}\leq\frac{\psi}{1-\rho}\Bigr[\|x_0\|_{2} + \sum_{t=0}^{T-2} (\|w_t\|_{2} +  \|B\|_2\|u_t\|_{2})\Bigr]:= X.
    \end{align*}
{Since the filtration $\mathcal{F}_t$ can alternatively be written as $\bm{\sigma}\{x_0, \{(u_\tau, w_\tau)\}_{\tau=0}^{t-1}\}$, sequentially applying the tower rule over $\mathcal{F}_t$ establishes that $X-\mathbb{E}[X]$ is sub-Gaussian with norm $\frac{\psi(\eta\sqrt{n}+\sigma\sqrt{m}\|B\|_2)\sqrt{T}}{1-\rho}$.} 
   Then, Hoeffding's inequality yields
    \begingroup\fontsize{9.5pt}{11.5pt}\selectfont\begin{align*}
        &\mathbb{P}\Bigr( X-\mathbb{E}[X] \leq \frac{(\eta{\sqrt{n}}+\sigma \sqrt{m}\cdot \|B\|_2)T}{1-\rho}   \Bigr)  \geq 1-\exp(-\Omega(T)).\numberthis\label{xt123}
    \end{align*}
    \endgroup%
 Considering $\sum_{t=0}^{T-1} \|x_t\|_2 \le X$ and applying the expectation bound 
    $\mathbb{E}\bigr[\sum_{t=0}^{T-1} \|x_t\|_2\bigr]=\mathcal{O}\bigr(\frac{(\eta{\sqrt{n}}+\sigma \sqrt{m}\cdot \|B\|_2)T}{1-\rho}\bigr)$ from \eqref{xtsums} to \eqref{xt123} completes the proof. 
\end{proof}

\subsection{Proof of Lemma \ref{minrank1}}\label{minrank1proof}
{\begin{proof}
Introduce an auxiliary Gaussian vector $g \sim N(0, I_{r})$. Since, for any $x \in \mathbb{R}^{r}$,
$\mathbb{E}_g [|g^T x|] = \sqrt{\frac{2}{\pi}} \|x\|_2$,
we have
\[
\begin{aligned}
&\sqrt{\frac{2}{\pi}}\sum_{t=k-1}^{T+k-2} \mathbb{I}_\pm \{v_t^{(s)}=0\} \|Z \mathbf{U}_t^{(k)}\|_2 \\&= \sum_{t=k-1}^{T+k-2} \mathbb{I}_\pm \{v_t^{(s)}=0\} \mathbb{E}_g [|g^T Z \mathbf{U}_t^{(k)}|] \\
&= \mathbb{E}_g \left[ \sum_{t=k-1}^{T+k-2} \mathbb{I}_\pm \{v_t^{(s)}=0\} |(Z^T g)^T \mathbf{U}_t^{(k)}| \right] = \mathbb{E}_g [h_T(Z^T g)].
\end{aligned}
\]
By the definition of $m_T$ and positive homogeneity,
\[
h_T(z) \ge m_T \|z\|_2 \qquad \forall z \in \mathbb{R}^{mk}.
\]
Thus, we arrive at
\begin{equation} \label{eq:ft_lower_bound}
\sum_{t=k-1}^{T+k-2} \mathbb{I}_\pm \{v_t^{(s)}=0\} \|Z \mathbf{U}_t^{(k)}\|_2 \ge \sqrt{\frac{\pi}{2}}m_T \mathbb{E}_g [\|Z^T g\|_2].
\end{equation}

Let $s_1, \ldots, s_r$ be the singular values of $Z$ with $\sum_j s_j^2 = \|Z\|_F^2 = 1$. By strict concavity of $\sqrt{\cdot}$, Jensen's inequality yields
\[
\mathbb{E}_g [\|Z^T g\|_2 ]= \mathbb{E}\sqrt{\sum_{j=1}^r s_j^2 g_j^2} \ge \sum_{j=1}^r s_j^2 \mathbb{E}[|g_j|] 
= \sqrt{\frac{2}{\pi}},
\]
where equality holds if and only if exactly one $s_j > 0$ (i.e., $\operatorname{rank}(Z)=1$). Thus,
\[
\mathbb{E}_g [\|Z^\top g\|_2 ]
\begin{cases}
= \sqrt{2/\pi}, & \operatorname{rank}(Z) = 1, \\
> \sqrt{2/\pi}, & \operatorname{rank}(Z) > 1.
\end{cases}
\]
Combining with \eqref{eq:ft_lower_bound}, since $m_T > 0$,
\[
\sum_{t=k-1}^{T+k-2} \mathbb{I}_\pm \{v_t^{(s)}=0\} \|Z \mathbf{U}_t^{(k)}\|_2 \ge m_T,
\]
with strict inequality for every matrix of rank greater than one.

Now, let
$z^* \in \argmin_{\|z\|_2=1} h_T(z)$
and choose any  vector $s \in \mathbb{R}^{r}$ such that $\|s\|_2=1$. Consider
$Z^* = s (z^*)^T$ which satisfies  $\|Z^*\|_F = 1$ and $\operatorname{rank}(Z^*) = 1$. Then, we have 
    \begin{align*}
    \|Z^*\mathbf{U}_t^{(k)}\|_2 = \|s(z^*)^T\mathbf{U}_t^{(k)}\|_2 = 
    |(z^*)^T \mathbf{U}_t^{(k)}|,
    \end{align*}
 which concludes that   
\[\sum_{t=k-1}^{T+k-2} \mathbb{I}_\pm \{v_t^{(s)}=0\} \|Z^* \mathbf{U}_t^{(k)}\|_2 
= m_T.\]
Thus, $Z^*$ is globally optimal, and all global minimizers have rank-1.
\end{proof}
}

\renewcommand*{\bibfont}{\footnotesize}
\printbibliography

@book{bertsekas1995dynamic,
  title={Dynamic Programming and Optimal Control, Two Volume Set},
  author={Bertsekas, Dimitri P},
  year={2012},
  publisher={Athena scientific, MA},
  edition = {4th}
}

@article{kalman1960new,
  title={A new approach to linear filtering and prediction problems},
  author={Kalman, Rudolf E.},
  journal={Journal of Basic Engineering},
  volume={82},
  number={1},
  pages={35--45},
  year={1960},
  publisher={ASME}
}

@inproceedings{yalcin2026subgradient,
  title={Subgradient Method for System Identification with Non-Smooth Objectives},
  author={Baturalp Yalcin and Jihun Kim and Javad Lavaei},
  booktitle={Learning for Dynamics and Control Conference},
  pages={2151--2170},
  volume = 	 {331},
  publisher =    {PMLR},
  year={2026}
}

@inproceedings{kim2025system,
  title={System Identification from Partial Observations under Adversarial Attacks},
  author={Jihun Kim and Javad Lavaei},
  booktitle={IEEE Conference on Decision and Control (CDC)},
  year={2025},
pages={1384--1391}
}

@article{zhang2024exact,
  title={Exact Recovery Guarantees for Parameterized Nonlinear System Identification Problem under Sparse Disturbances or Semi-Oblivious Attacks},
  author={Haixiang Zhang and Baturalp Yalcin and Javad Lavaei and Eduardo D. Sontag},
journal={Transactions on Machine Learning Research},
  year={2025},
 publisher={JMLR}
}

@book{zhourobust,
  title={Robust and optimal control},
  author={Kemin Zhou and John C. Doyle and Keith Glover},
  year={1996},
  publisher={Prentice Hall}
}

@article{bottou2018optimization,
  title={Optimization methods for large-scale machine learning},
  author={Bottou, L{\'e}on and Curtis, Frank E and Nocedal, Jorge},
  journal={SIAM Review},
  volume={60},
  number={2},
  pages={223--311},
  year={2018},
  publisher={SIAM}
}

@book{boyd2004convex,
  title={Convex Optimization},
  author={Stephen Boyd and Lieven Vandenberghe},
  year={2004},
  publisher={Cambridge University Press}
}

@book{nesterov2004intro,
  title={Introductory Lectures on Convex Optimization: A Basic Course},
  author={Yurii Nesterov},
  year={2004},
  publisher={Springer New York}
}

@book{bubeck2015convex,
  title={Convex Optimization: Algorithms and Complexity},
  author={Bubeck, S{\'e}bastien},
  year={2015},
  publisher={Now Publishers Inc.},
  series={Foundations and Trends{\textregistered} in Machine Learning},
  volume={8},
  number={3-4},
  pages={231--357}
}

@article{sarkar2021finite,
  title={Finite Time {LTI} System Identification},
  author={Tuhin Sarkar and Alexander Rakhlin and Munther A. Dahleh},
  journal={Journal of Machine Learning Research},
  volume={22},
  number={26},
  pages={1--61},
  year={2021}
}

@inproceedings{sarkar2019near,
  title={Near optimal finite time identification of arbitrary linear dynamical systems},
  author={Tuhin Sarkar and Alexander Rakhlin},
  booktitle={International Conference on Machine Learning},
  pages={5610--5618},
  year={2019}
}

@inproceedings{oymak2019stochastic,
  title     = {Stochastic Gradient Descent Learns
State Equations with Nonlinear Activations},
  author    = {Samet Oymak},
  booktitle = {Annual Conference on Learning Theory},
  pages     = {1--29},
  year      = {2019},
publisher={PMLR}
}

@inproceedings{jain2021sgdrer,
  title     = {Streaming Linear System Identification with Reverse Experience Replay},
  author    = {Jain, Prateek and Kowshik, Suhas and Nagaraj, Dheeraj and Netrapalli, Praneeth},
  booktitle = {Advances in Neural Information Processing Systems},
  pages     = {30140--30152},
  year      = {2021}
}

@article{hardt2018gradient,
  title={Gradient descent learns linear dynamical systems},
  author={Moritz Hardt and Tengyu Ma and Benjamin Recht},
  journal={Journal of Machine Learning Research},
  volume={19},
  number={29},
  pages={1--44},
  year={2018},
  publisher={PMLR}
}

@article{kalman1961new,
  title={New results in linear filtering and prediction theory},
  author={Rudolf E. Kalman and Richard S. Bucy},
  journal={Journal of Basic Engineering},
  volume={83},
  number={1},
  pages={95--108},
  year={1961},
  publisher={ASME}
}

@article{hokalman1966kalman,
  title={Effective construction of linear state-variable models from input/output functions},
  author={Bin-Lun Ho and Rudolf E. Kalman},
  journal={Automatisierungstechnik},
  volume={14},
  number={112},
  pages={545--548},
  year={1966},
  publisher={De Gruyter}
}

@book{vershynin2025high,
  title={High-Dimensional Probability: An Introduction with Applications in Data Science},
  author={Roman Vershynin},
  year={2025},
    edition={2},
  publisher={Cambridge University Press}
}

@inproceedings{jedra2020finite,
  title={Finite-time Identification of Stable Linear Systems Optimality of the Least-Squares Estimator},
  author={Jedra, Yassir and Proutiere, Alexandre},
  booktitle={IEEE Conference on Decision and Control (CDC)},
  year={2020},
pages={996--1001}
}

@inproceedings{simchowitz2018learningwithout,
  title={Learning Without Mixing: Towards A Sharp Analysis of Linear System Identification},
  author={Max Simchowitz and Horia Mania and Stephen Tu and Michael I. Jordan and Benjamin Recht},
  booktitle={Conference On Learning Theory},
  pages={439--473},
  year={2018}
}

@article{fawzi2014secure,
  title={Secure Estimation and Control for Cyber-Physical Systems Under Adversarial Attacks},
  author={Hamza Fawzi and Paulo Tabuada and Suhas Diggavi},
  journal={IEEE Transactions on Automatic Control},
  volume={59},
  number={6},
  pages={1454--1467},
  year={2014},
  publisher={IEEE}
}

@article{shoukry2016event,
  title={Event-Triggered State Observers for Sparse Sensor Noise/Attacks},
  author={Yasser Shoukry and Paulo Tabuada},
  journal={IEEE Transactions on Automatic Control},
  volume={61},
  number={8},
  pages={2079--2091},
  year={2016},
  publisher={IEEE}
}

@article{pajic2017attack,
  title={Attack-Resilient State Estimation for Noisy Dynamical Systems},
  author={Miroslav Pajic and Insup Lee and George J. Pappas},
  journal={IEEE Transactions on Control of Network Systems},
  volume={4},
  number={1},
  pages={82--92},
  year={2017},
  publisher={IEEE}
}

@article{kim2024landscape,
  title={The Landscape of Deterministic and Stochastic Optimal Control Problems: One-Shot Optimization Versus Dynamic Programming},
  author={Jihun Kim and Yuhao Ding and Yingjie Bi and Javad Lavaei},
  journal={IEEE Transactions on Automatic Control},
  volume={69},
  number={12},
  pages={8587--8602},
  year={2024},
  publisher={IEEE}
}

@article{candes2005decoding,
  title={Decoding by Linear Programming},
  author={Emmanuel J. Cand\`{e}s and Terence Tao},
journal={IEEE Transactions on Information Theory},
  year={2005},
volume={51},
number={12},
pages={4203--4215}
}

@article{oymak2022revisit,
  title={Revisiting {H}o–{K}alman-Based System Identification: Robustness and Finite-Sample Analysis},
  author={Samet Oymak and Necmiye Ozay},
  journal={IEEE Transactions on Automatic Control},
  volume={67},
  number={4},
  pages={1914--1928},
  year={2022},
  publisher={IEEE}
}

@article{pasqualetti2013cyber,
  title={Attack Detection and Identification in Cyber-Physical Systems},
  author={Fabio Pasqualetti and Florian D\"{o}rfler and Francesco Bullo},
  journal={IEEE Transactions on Automatic Control},
  volume={58},
  number={11},
  pages={2715--2729},
  year={2013},
  publisher={IEEE}
}

@article{li2023disturbance,
  title={Disturbance decoupled secure state estimation: An orthogonal projection-based method},
  author={Jinghao Li and Guang-Hong Yang},
  journal={Automatica},
  volume={147},
  articleno = {110740},
  year={2023},
  publisher={Elsevier}
}

@inproceedings{kim2024prevailing,
  title={Prevailing against Adversarial Noncentral Disturbances: Exact Recovery of Linear Systems with the $ \ell_1$-norm Estimator},
  author={Kim, Jihun and Lavaei, Javad},
  booktitle={American Control Conference (ACC)},
  pages={1161--1168},
  year={2025}
}

@inproceedings{kim2026sharp,
  title={On the Sharp Input-Output Analysis of Nonlinear Systems under Adversarial Attacks},
  author={Jihun Kim and Yuchen Fang and Javad Lavaei},
  booktitle={International Conference on Machine Learning},
   pages={57445--57473},
  year={2026}
}

@article{kim2025revisit,
  title={Revisiting the Geometrically Decaying Step Size: Linear Convergence for Smooth or Non-Smooth Functions},
  author={Jihun Kim},
  journal={arXiv preprint arXiv:2508.13569},
  year={2025}
}

@article{kim2026defense,
  title={Defense Filters for Robust Learning of Partially Observed Dynamical Systems},
  author={Jihun Kim and Javad Lavaei},
  journal={arXiv preprint arXiv:2610.21935},
  year={2026}
}

@article{xie2017secure,
  author  = {Xie, Chun-Hua and Yang, Guang-Hong},
  title   = {Secure Estimation for Cyber-Physical Systems
             Under Adversarial Actuator Attacks},
  journal = {IET Control Theory \& Applications},
  volume  = {11},
  number  = {17},
  pages   = {2939--2946},
  year    = {2017}
}

@article{showkatbakhsh2020securing,
  author  = {Mehrdad Showkatbakhsh and
             Yasser Shoukry and Suhas N. Diggavi
             and Paulo Tabuada},
  title   = {Securing State Reconstruction under
             Sensor and Actuator Attacks:
             Theory and Design},
  journal = {Automatica},
  volume  = {116},
  pages   = {108920},
  year    = {2020}
}

@inproceedings{cardenas2008secure,
  author    = {C{\'a}rdenas, Alvaro A. and Amin, Saurabh and Sastry, Shankar},
  title     = {Secure Control: Towards Survivable Cyber-Physical Systems},
  booktitle = {2008 28th International Conference on Distributed Computing Systems Workshops},
  year      = {2008},
  pages     = {495--500},
  publisher = {IEEE}
}

@article{yalcin2024exact,
  title={Exact recovery for system identification with more corrupt data than clean data},
  author={Yalcin, Baturalp and Zhang, Haixiang and Lavaei, Javad and Arcak, Murat},
  journal={IEEE Open Journal of Control Systems},
  year={2024},
  publisher={IEEE}
}


\begin{IEEEbiography}[{\includegraphics[width=1in,height=1.25in,clip,keepaspectratio]{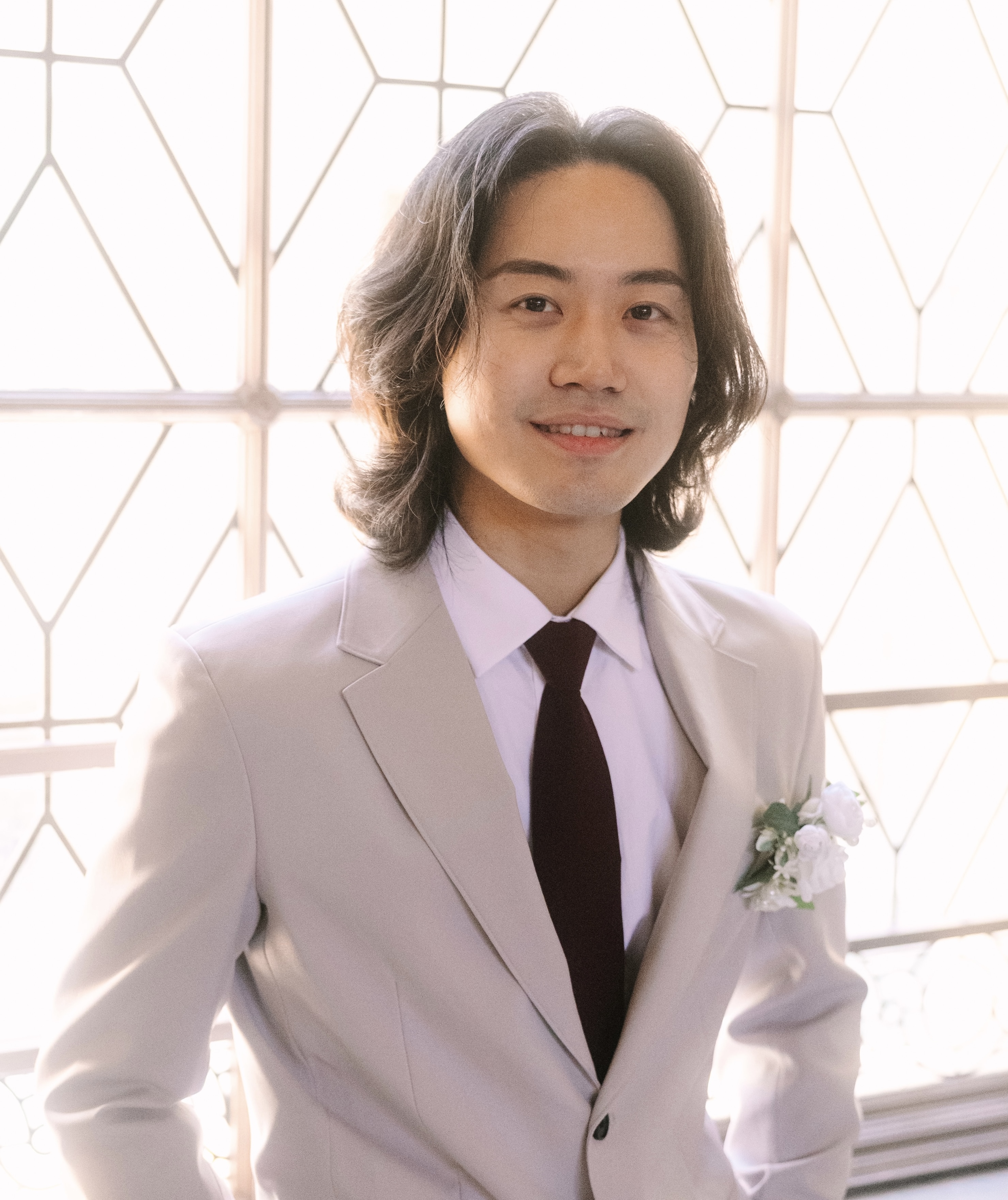}}]{Jihun Kim} (Graduate Student Member, IEEE) is a Ph.D. Candidate in Industrial Engineering and Operations Research at University of California, Berkeley, CA, USA. He obtained dual B.S. degrees in Industrial Engineering and Statistics from Seoul National University in 2022. His research interests include theoretical foundations in control, optimization, and machine learning to build trustworthy, sample-efficient, and robust approaches for real-world safety-critical systems in adversarial environments.
\end{IEEEbiography}

\begin{IEEEbiography}[{\includegraphics[width=1in,height=1.25in,clip,keepaspectratio]{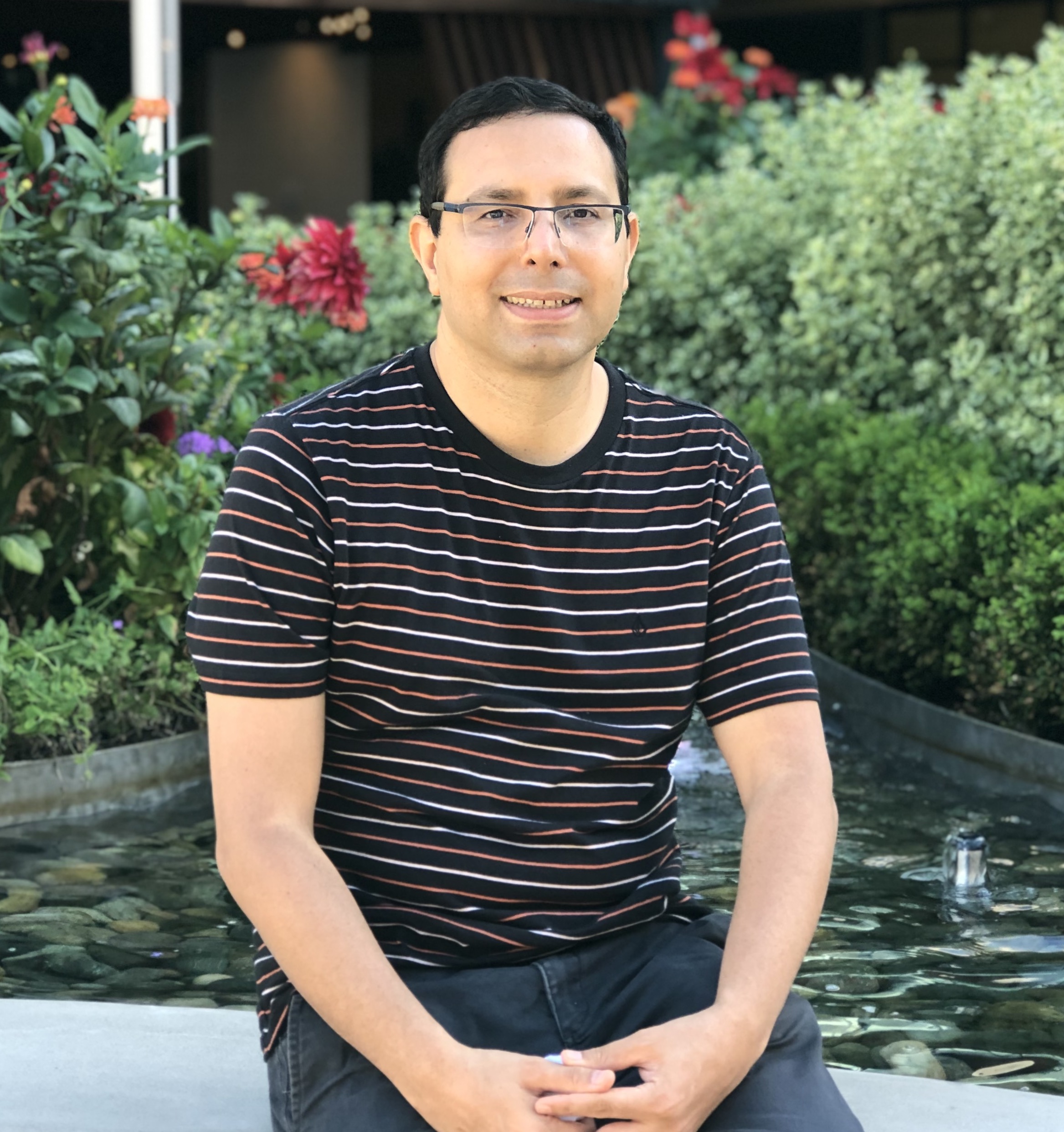}}]{Javad Lavaei} (Fellow, IEEE) is an Associate Professor in the Department of Industrial Engineering and Operations Research at UC Berkeley. He obtained the Ph.D. degree in Control \& Dynamical Systems from California Institute of Technology. He is a senior editor of the IEEE Systems Journal and has served on the editorial boards of the IEEE Transactions on Automatic Control, IEEE Transactions on Control of Network Systems, IEEE Transactions on Smart Grid, and IEEE Control Systems Letters. 
\end{IEEEbiography}

\end{document}